\documentclass[10pt,a4paper]{article}
\usepackage{amsfonts,amssymb,amsmath,amscd,latexsym,makeidx,amsthm, color,hyperref,graphicx,multicol,pdfsync}
\usepackage[latin1]{inputenc}
\usepackage[english]{babel}
\usepackage{comment}
\usepackage{accents}
\numberwithin{equation}{section}

\newtheorem{trm}{Theorem}[section]
\newtheorem{prop}[trm]{Proposition}

\newtheorem{lemma}[trm]{Lemma}

\newtheorem{rem}[trm]{Remark}

\newcommand{\W}{\mathcal W}
\newcommand{\ra}{\longrightarrow}
\newcommand{\rw}{\rightharpoonup}
\newcommand{\dis}{\displaystyle}

\newcommand{\ov}{\overline}

\newcommand{\R}{\mathbb{R}}

\newcommand{\N}{\mathbb{N}}

\newcommand{\eps}{\varepsilon}

\newcommand{\0}{\emptyset}
\newcommand{\ph}{\varphi}

\newcommand{\bra}[1]{\left({#1}\right)}
\newcommand{\sqbra}[1]{\left[{#1}\right]}
\newcommand{\cur}[1]{\left\{#1\right\}}

\newcommand{\lm}{(-\Delta)^m}

\newlength{\dhatheight}

\DeclareMathOperator{\dv}{div}

\newenvironment{Si}[1]{\left\{\begin{array}{#1}}{\end{array} \right. }

\makeindex 
 
 \author{Azahara DelaTorre$\phantom{}^*$  \\ \small Università di Padova \\ \footnotesize \texttt{azahara.delatorre@math.unipd.it} \and Gabriele Mancini\thanks{The authors are supported by Swiss National Science Foundation, projects nr. PP00P2-144669 and PP00P2-170588/1.} \\ \small Università di Padova \\ \footnotesize \texttt{gabriele.mancini@math.unipd.it}  }
\title{Improved Adams-type  inequalities and their extremals in dimension $2m$}

\begin{document}
 \maketitle
\begin{abstract}
In this paper we prove the existence of extremal functions for the Adams-Moser-Trudinger inequality on the Sobolev space $H^{m}(\Omega)$, where $\Omega$ is any bounded, smooth, open subset of $\R^{2m}$, $m\ge 1$. Moreover, we extend this result to improved versions of Adams' inequality of Adimurthi-Druet type. Our strategy is based on blow-up analysis for sequences of subcritical extremals and introduces several new techniques and constructions. The most important one is a new procedure for obtaining capacity-type estimates on annular regions. 
\end{abstract}

 \section{Introduction}
Given $m\in \N$, $m\ge 1$, let $\Omega\subseteq \R^{2m}$ be a bounded open set with smooth boundary. For any $\beta >0$, we  consider the Moser-Trudinger functional 
$$
F_{\beta}(u) := \int_{\Omega} e^{\beta u^2}dx
$$
and the set 
 $$
M_{0}:= \cur{u\in H^m_0(\Omega)\: :\; \|u\|_{H^m_0(\Omega)}\le 1},
 $$
where 
$$
\|u\|_{H^m_0(\Omega)} = \|\Delta^{\frac{m}{2}} u\|_{L^2(\Omega)}  \qquad \mbox{ and }\qquad 
\Delta^{\frac{m}{2}} u  := \begin{Si}{cl} 
\Delta^{n} u & \mbox{ if } m = 2n, \;n\in \N, \\
\nabla\Delta^n u &  \mbox{ if } m = 2n +1,\; n\in \N.
\end{Si}
$$
The Adams-Moser-Trudinger inequality (see \cite{Adams}) implies that 
\begin{equation}\label{Adams}
\sup_{M_{0}} F_\beta <+\infty \qquad \Longleftrightarrow \qquad \beta \le \beta^*,
\end{equation}
where $\beta^*:= m (2m-1)! Vol(\mathbb S^{2m})$.  This result is an extension to dimension $2m$ of the work done by  Moser \cite{mos} and Trudinger \cite{tru} in the case $m=1$, and can be considered as a critical version of the Sobolev inequality for the space $H^m_0(\Omega)$. A classical problem related to Moser-Trudinger and Sobolev-type embeddings consists in investigating the existence of extremal functions. While it is rather simple to prove that the supremum in \eqref{Adams} is attained for any $\beta <\beta^*$,  lack of compactness due to concentration phenomena makes the critical case $\beta = \beta^*$ challenging.  The first proof of existence of extremals for \eqref{Adams} was given by Carleson and Chang  \cite{CC} in the special setting $m=1$ and $\Omega = B_1(0)$. The case of arbitrary domains $\Omega\subseteq \R^2$ was treated by Flucher in \cite{flu}. These results are based on sharp estimates on the values that $F_\beta$ can attain on concentrating sequences of functions. Recently, a different approach was proposed in \cite{MarMan} and \cite{DizDru}. Concerning the higher order case, as far as we know, the existence of extremals was proved only for $m=2$ by  Lu and Yang in \cite{LuYa2} (see also \cite{Ndi}). In this work, we are able to study the problem for any arbitrary $m\ge 1$. Indeed, we prove here the following result.

\begin{trm}
Let $\Omega\subseteq \R^{2m}$ be a smooth bounded domain, then for any $m\ge 1$ and $\beta\le \beta^*$ the supremum in \eqref{Adams} is attained, i.e. there exists a function $u^*\in M_0$ such that $F_{\beta}(u^*) = \sup_{M_0} F_\beta$. 
\end{trm}

More generally, we are interested in studying extremal functions for a larger family of inequalities. Let us denote $$
\lambda_1(\Omega):= \inf_{u\in H^m_0(\Omega), u\neq 0} \frac{\|u\|^2_{H^m_0(\Omega)}}{\|u\|^2_{L^2(\Omega)}}.
$$ 
For the 2-dimensional case, in \cite{AD} it was proved that if $\Omega \subseteq \R^2$ and  $0\le \alpha<\lambda_1(\Omega)$, then
\begin{equation}\label{AD}
\sup_{u\in M_0} \int_{\Omega} e^{\beta^*u^2(1+\alpha\|u\|^2_{L^2(\Omega)})} dx <+\infty.
\end{equation}
Moreover the bound on $\alpha$ is sharp, i.e. the supremum is infinite for any $\alpha \ge\lambda_1(\Omega)$.  A stronger form of this inequality can be deduced from the results in \cite{Tint}:
\begin{equation}\label{ADstr}
\sup_{u\in H^1_0(\Omega),\; \|u\|_{H^1_0(\Omega)}^2-\alpha\|u\|_{L^2(\Omega)}^2 \le 1 } F_{\beta^*} <+\infty. 
\end{equation}
Surprisingly, the study  of extremals   for the stronger inequality \eqref{ADstr} is easier than for \eqref{AD}. In fact, it was proved in \cite{Yang} that the supremum in \eqref{ADstr} is attained for any $0\le \alpha<\lambda_1(\Omega)$, while existence of extremal functions for \eqref{AD} is known only for small values of $\alpha$ (see \cite{LuYa}). Such results have been extended to dimension 4 in \cite{LuYa2} and \cite{Ngu}. In this paper, we consider the case of an arbitrary  $m\ge 1$.    For any $0\le \alpha <\lambda_1(\Omega)$ we denote 
$$
\|u\|_{\alpha}^2 := \|u\|_{H^{m}_0(\Omega) }^2 - \alpha\|u\|_{L^2(\Omega)}^2,
$$
and we consider the set 
$$
M_{\alpha}:= \cur{u\in H^m_0(\Omega)\: :\; \|u\|_{\alpha}\le 1} 
$$
and the quantity 
\begin{equation}\label{sup}
S_{\alpha,\beta}:= \sup_{M_\alpha} F_\beta. 
\end{equation}
Observe that Poincare's inequality implies that for any $0\le \alpha<\lambda_1(\Omega)$, $\|\cdot\|_\alpha$ is a norm on $H^m_0$ which is equivalent to $\|\cdot\|_{H^m_0}$. Our main result is the following:

\begin{trm}\label{main} 
Let $\Omega\subseteq \R^{2m}$ be a smooth bounded domain, then for any $m\ge 1$ the following holds:
\begin{enumerate}
\item For any $0\le \beta \le \beta^*$ and $0\le \alpha <\lambda_1(\Omega)$ we have $S_{\alpha,\beta}<+\infty$, and there exists a function $u^*\in M_\alpha$ such that $F_\beta(u^*)=S_{\alpha,\beta}$.
\item If $\alpha \ge \lambda_1(\Omega)$, or $\beta > \beta^*$, we have $S_{\alpha,\beta}=+\infty$. 
\end{enumerate}
\end{trm}

The proof of the first part of Theorem \ref{main} for $\beta = \beta^*$ is the most difficult one and it is based on blow-up analysis for sequences of sub-critical extremals. We will take a sequence $\beta_n \nearrow \beta^*$ and find $u_n \in M_\alpha$, such that $F_{\beta_n}(u_n) = S_{\alpha,\beta_n}$.  If $u_n$ is bounded in $L^\infty(\Omega)$, then standard elliptic regularity proves that $u_n $ converges in $H^m(\Omega)$ to a function $u_0 \in M_\alpha$ such that $F_{\beta^*}(u_0)= S_{\alpha,\beta^*}$. Hence, one has to exclude that $u_n$ blows-up, i.e. that $\dis{\mu_n:= \max_{\ov{\Omega}}|u_n|\to +\infty}$. This is done through a contradiction argument.  On the one hand, if $\mu_n \to +\infty$, one can show that $u_n$ admits a unique blow-up point $x_0$ and give a precise description of the behavior of $u_n$ around $x_0$. Specifically, we will prove (see Proposition \ref{big}) that blow-up implies
\begin{equation*}
S_{\alpha,\beta^*}= \lim_{n\to + \infty} F_{\beta_n}(u_n) \le |\Omega|+\frac{Vol(\mathbb S^{2m})}{2^{2m}} e^{ \dis{\beta^* \bra{C_{\alpha,x_0} -I_m}}},
\end{equation*}
where $C_{\alpha,x_0}$ is the value at $x_0$ of the trace of the regular part of the Green's function for the operator $\lm - \alpha$, and $I_m$ is a dimensional constant.   On the other hand, by exhibiting a suitable test function, we will prove (see Proposition \ref{proptest}) that such upper bound cannot hold, concluding the proof. 

 While the general strategy is rather standard in the study of this kind of problems (see e.g. \cite{AD}, \cite{flu}, \cite{IulMan}, \cite{Li1}, \cite{Li2}, \cite{LuYa},  \cite{LuYa2}, \cite{Ngu} and \cite{Yang}), our proof introduces several elements of novelty. 

First, our description of the behaviour of $u_n$ near its blow-up point $x_0$ is sharper than the one given for $m=2$ in \cite{LuYa2} and \cite{Ngu}. There,  in order to compensate the lack of sufficiently sharp standard elliptic estimates on a small scale, the authors needed to modify the standard scaling  for the Euler-Lagrange equation satisfied by $u_n$.    Instead, following the approach first introduced in \cite{mar}, we are able to use the standard scaling replacing classical elliptic estimates with Lorentz-Zygmund type regularity estimates.  

Secondly, in order to describe the behaviour of $u_n$ far from $x_0$, we  extend to higher dimension the approach of Adimurthi and Druet \cite{AD}, which is based on the properties of truncations of $u_n$. To preserve the high-order regularity required in the high-dimensional setting, we introduce polyharmonic truncations. This step, requires precise pointwise estimates on the derivatives of $u_n$, which  are a generalisation of the ones in \cite{MS}, where the authors study sequences of positive critical points of $F_\beta$ constrained to spheres in  $H^m_0$. We stress that the results of \cite{MS} cannot be directly applied to our case, since here subcritical maximizers are not necessarily positive in $\Omega$ if $m\ge 2$. In addition,  the presence of the parameter $\alpha$ modifies the Euler-Lagrange equation.  While the differences in the nonlinearity do not create significant issues, the argument in \cite{MS} relies strongly on the positivity assumption. Therefore, here we propose a different proof. 

The most important feature of our proof of Theorem \ref{main} is that it does not rely on explicit capacity estimates. A crucial step in our blow-up analysis consists in finding sharp lower bounds for the integral of $|\Delta^\frac{m}{2}u_n|^2$ on annular regions.  In all the earlier works, this is achieved by comparing the energy of $u_n$ with the quantity
$$
i(a,b,R_1,R_2):= \min_{u\in E_{a,b}} \int_{ \cur{R_1 \le |x| \le R_2 }}  |\Delta^\frac{m}{2}u|^2 dy 
$$
for suitable choices of $a=(a_0,\ldots,a_{m-1})$, $b=(b_0,\ldots,b_{m-1})$, and where $E_{a,b}$ denotes the set of all the $H^m$ functions on  $\cur{R_1 \le |x|  \le R_2 } $ satisfying  $\partial_\nu^i u_n = a_i$ on $\partial B_{R_1}(0)$ and  $\partial_\nu^i u_n = b_i$ on $\partial B_{R_2}(0)$ for $i =0,\ldots, m-1$. While for $m=1$ or $m=2$, $i(a,b,R_1,R_2)$ can be explicitly computed, finding its expression for an arbitrary $m$ appears to be very hard. In our work we show that these capacity estimates are unnecessary, since equivalent lower bounds can be obtained by  directly comparing the Dirichlet energy of $u_n$ with the energy of a suitable polyharmonic function. This results in a considerable simplification of the proof, even for $m=1,2$. 

Finally, working with arbitrary values of $m$ makes much harder the construction of good test functions and the study of blow-up near $\partial \Omega$, since standard moving planes techniques are not available for $m\ge 2$. To address the last issue, we will apply the Pohozaev-type identity introduced in \cite{RobWei} and applied in \cite{MarPet} to Liouville-type equations. 

It would be interesting to extend our result to Adams' inequality in odd dimension or, more generally, to the non-local Moser-Trudinger inequality for fractional-order Sobolev spaces  proved in \cite{Marfraz}, for which the existence of extremals is still open. In this fractional setting, the behavior of  blowing-up subcritical extremals was studied in \cite{MarSch} (at least for nonnegative functions). However, obtaining capacity-type estimates becomes much more challenging, and our argument to avoid them relies strongly on the local nature of the operator $\lm$.    

This paper is organized as follows. In Section \ref{prel}, we will introduce some notation and state some preliminary results.  In Section \ref{subcrit}, we will focus on the subcritical case $\beta<\beta^*$.  In Section \ref{main sec}, we will analyze the blow up behavior of subcritical extremals. Since  this part of the paper  will discuss the most important elements of our work, it will be divided into several subsections. Finally, in Section \ref{sec test}, we will introduce new test functions and we will complete the proof of Theorem \ref{main}. For the reader convenience, we will recall in Appendix  some known results concerning elliptic estimates for the operator $\lm$.

\section*{Acknowledgments}
 
We  are grateful  to Professor Luca Martinazzi for introducing us to the problem and for supporting us in the preparation of this work with his encouragement and with many invaluable suggestions. 

A consistent part of this work was carried out while we were employed by the University of Basel. We would like to thank the Department of Mathematics and Computer Science for their hospitality and support.

\section{Preliminaries}\label{prel}
Throughout the paper we will denote by $\omega_l$ the $l-$dimensional Hausdorff measure of the unit sphere $\mathbb S^{l}\subseteq \R^{l+1}.$ We recall that, for any $m\ge 1$, 
\begin{equation}\label{omega}
\omega_{2m-1} = \frac{2\pi^m}{(m-1)!}  \qquad \mbox{ and } \qquad   \omega_{2m} = \frac{2^{m+1}\pi^m}{(2m-1)!!} .
\end{equation}

It is known that the fundamental solution of $\lm$ in $\R^{2m}$ is given by $-\frac{1}{\gamma_m} \log|x|$, where 
$$
\gamma_m:= \omega_{2m-1} 2^{2m-2} [(m-1)!]^2 
= \frac{\beta^*}{2m},
$$
with $\beta^*$ defined as in \eqref{Adams}. In other words, one has 
$$
\lm \bra{-\frac{2m}{\beta^*} \log|x| }=  \delta_0  \qquad \mbox{ in } \R^{2m}.
$$
More generally, for any $1 \le l \le m-1$, we have 
$$
\Delta^l (\log|x|)  =\tilde  K_{m,l} \frac{1}{|x|^{2l}},
$$
where 
\begin{equation}\label{Kml1}
\begin{split}
\tilde K_{m,l} & 
 = (-1)^{l+1} 2^{2l-1}  \frac{(l-1)!(m-1)!}{(m-l-1)! }.
\end{split}\end{equation}
This also yields 
$$
\Delta^{l+\frac{1}{2}} (\log|x|)  = -2l \tilde K_{m,l} \frac{x}{|x|^{2l+2}}.
$$
For any $1\le j\le 2m-1$, we define 
\begin{equation}\label{Kml2}
K_{m,\frac{j}{2}}:= \begin{Si}{cl}
\tilde K_{m,\frac{j}{2}} &   \text{ for } j \text{ even }\\ 
\rule{0cm}{0.5cm} -(j-1) \tilde K_{m,\frac{j-1}{2}} & \text{ for }  j \text{ odd}, j\ge 3, \\ 
1  & \text{ for }   j= 1.
\end{Si}
\end{equation}
Then, we obtain
\begin{equation}\label{ej}
\Delta^\frac{j}{2} (\log |x|) = \frac{K_{m,\frac{j}{2}}}{|x|^j} e_j(x),   \qquad \text{ where }\qquad e_j(y):= \begin{Si}{cc}
1 & j \text{ even,}\\
\frac{y}{|y|} & j\text{ odd}.
\end{Si}
\end{equation}

In order to use the same notation for all the values of $m$, we will use the symbol $\cdot $ to denote  both the scalar product between vectors in $\R^{2m}$ and the standard Euclidean product between reals numbers. This turns out to be very useful to have compact integration by parts formulas. For instance, we will use several times the following Proposition: 

\begin{prop}\label{parts}
Let $\Omega \subseteq \R^{2m}$ be a bounded open domain with Lipschitz boundary. Then, for any $u\in H^m(\Omega)$, $v\in H^{2m}(\Omega)$, we have 
$$
\int_{\Omega}\Delta^\frac{m}{2} u \cdot \Delta^\frac{m}{2} v\,  dx = \int_{\Omega} u \lm v \, dx -\sum_{j=0}^{m-1} \int_{\partial \Omega} (-1)^{m+j} \nu \cdot \Delta^{\frac{j}{2}} u \, \Delta^{\frac{2m-j-1}{2}} v  \, d\sigma,
$$
where $\nu$ denotes the outer normal to $\partial \Omega$. 
\end{prop}

A crucial role in our proof will be played by  Green's functions for operators of the form $\lm -\alpha$. We recall here that for any $x_0\in \Omega$, and  $0\le \alpha <\lambda_1(\Omega)$, there exists a unique distributional solution $G_{\alpha,x_0}$ of 
\begin{equation}\label{green}
\begin{Si}{cc}
\lm G_{\alpha,x_0}  = \alpha G_{\alpha,x_0} + \delta_{x_0}  & \mbox{ in }\Omega,\\
G_{\alpha,x_0} = \partial_\nu G_{\alpha,x_0} = \ldots = \partial^{m-1}_{\nu} G_{\alpha,x_0}  = 0    & \mbox{ on } \partial \Omega.
\end{Si}
\end{equation}

Some of the main properties of the function $G_{\alpha,x_0}$ are listed in the following Proposition.  We refer to \cite{Boggio} and \cite{AcSw} for the proof of the case $\alpha=0$, while the general case can be obtained with minor modifications.

\begin{prop}\label{prop green}  Let $\Omega$ be a bounded open set with smooth boundary. Then, for any $x_0\in \Omega$ and $0\le \alpha < \lambda_1(\Omega)$, we have:
\begin{enumerate}
\item There exist $C_{\alpha,x_0} \in \R$ and  $\psi_{\alpha,x_0} \in C^{2m-1}(\ov{\Omega})$ such that  $\psi_{\alpha,x_0}(x_0)=0$ and
$$G_{\alpha,x_0}(x)= -\frac{2m}{\beta^*}\log|x-x_0| + C_{\alpha,x_0} + \psi_{\alpha,x_0}(x), \qquad  \text{ for any } x\in \Omega \setminus \{x_0\}.$$  
\item There exists a constant $C = C(m,\alpha,\Omega)$ independent of $x_0$,  such that 
$$
|G_{\alpha, x_0}(x)| \le C |\log |x-x_0||,
$$ 
and 
$$
|\nabla^l G_{\alpha,x_0}(x)|  \le \frac{C}{|x-x_0|^l},$$
for any 
$ 1\le l  \le 2m-1, x \in \Omega \setminus \{x_0\}.$ 
\item $G_{\alpha,x_0}(x)= G_{\alpha,x}(x_0)$, for any $x\in \Omega \setminus \{x_0\}.$
\end{enumerate}
\end{prop}

In addition, using integration by parts and  Proposition \ref{prop green},   we can establish the following new property.   

\begin{lemma}\label{int Green}
For any $x_0\in \Omega$ and $0\le \alpha<\lambda_1(\Omega)$, we have
$$
\int_{\Omega\setminus B_\delta(x_0)} |\Delta^\frac{m}{2} G_{\alpha,x_0} |^2 dx  = \alpha \|G_{\alpha,x_0} \|_{L^2(\Omega)}^2 - \frac{2m }{\beta^*} \log \delta + C_{\alpha,x_0} +H_m + O(\delta |\log \delta|),
$$
as $\delta \to 0$, where $C_{\alpha, x_0}$ is as in Proposition \ref{prop green} and
\begin{equation}\label{Hm}
H_m := \begin{Si}{cc}
\dis{\bra{\frac{2m}{\beta^*} }^2   \omega_{2m-1}  \sum_{j=1}^{m-1} (-1)^{j+m} K_{m,\frac{j}{2}} K_{m,\frac{2m-j-1}{2}}} & \mbox{ if } m\ge 2,\\
0  & \mbox{ if } m=1. 
\end{Si}
\end{equation}
\end{lemma}
\begin{proof}
From  Proposition \ref{parts} applied in $\Omega \setminus B_\delta(x_0)$ and \eqref{green}, we find 
$$
\int_{\Omega\setminus B_\delta(x_0)} |\Delta^\frac{m}{2} G_{\alpha,x_0} |^2 dx  =  \alpha \int_{\Omega\setminus B_\delta(x_0)} G_{\alpha,x_0}^2 dx + \sum_{j=0}^{m-1}\int_{\partial B_\delta(x_0)} (-1)^{m+j} \nu \cdot \Delta^{\frac{j}{2}} G_{\alpha,x_0} \Delta^{\frac{2m-j-1}{2}} G_{\alpha,x_0} \, d\sigma.
$$ 
On $\partial B_\delta(x_0)$, Proposition \ref{prop green}, \eqref{ej},  and the identity $\frac{2m}{\beta^*}K_{m,\frac{m-1}{2}}= \frac{(-1)^{m-1}}{\omega_{2m-1}} $ yield  
\[\begin{split}
\nu \cdot G_{\alpha,x_0} \Delta^{\frac{2m-1}{2}} G_{\alpha,x_0}   & =\bra{-\frac{2m}{\beta^*} \log \delta + C_{\alpha,x_0} +O(\delta)} \bra{\frac{-2m}{\beta^*}K_{m,\frac{2m-1}{2}}\delta^{1-2m}+O(1)} \\
&= \frac{(-1)^m  }{\omega_{2m-1}} \delta^{1-2m}  \bra{-\frac{2m}{\beta^*}\log \delta +C_{\alpha,x_0} + O(\delta) + O(\delta^{2m-1} |\log \delta|)},
\end{split}
\]
and, for $m\ge 2$ and $1\le j \le m-1$, that
\[\begin{split}
 \nu \cdot \Delta^{\frac{j}{2}} G_{\alpha,x_0} \Delta^{\frac{2m-j-1}{2}} G_{\alpha,x_0}  &= \bra{-\frac{2m}{\beta^*}K_{m,\frac{j}{2}}\delta^{-j}+O(1)}  \bra{-\frac{2m }{\beta^*} K_{m,\frac{2m-j-1}{2}} \delta^{1+j-2m}+ O(1)} \\
& = \bra{\frac{2m}{\beta^*}}^2 K_{m,\frac{j}{2}} K_{m,\frac{2m-j-1}{2}} \delta^{1-2m} (1+O(\delta^j)).
\end{split}
\]
Then, we get 
\begin{equation}\label{int boun}
\sum_{j=0}^{m-1}\int_{\partial B_\delta(x_0)} (-1)^{m+j} \nu \cdot \Delta^{\frac{j}{2}} G_{\alpha,x_0} \Delta^{\frac{2m-j-1}{2}} G_{\alpha,x_0} \, d\sigma = - \frac{2m }{\beta^*} \log \delta + C_{\alpha,x_0} +H_m + O(\delta |\log \delta|),
\end{equation}
with $H_m$ as in \eqref{Hm}. 
Finally, applying again Proposition \ref{prop green}, 
we find 
\begin{equation}\label{l2 part}
\int_{\Omega\setminus B_\delta(x_0)} G_{\alpha,x_0}^2 dx = \|G_{\alpha,x_0} \|_{L^2(\Omega)}^2 + O(\delta^{2m}\log^2 \delta).
\end{equation}
The conclusion follows by \eqref{int boun} and \eqref{l2 part}. 
\end{proof}

\begin{rem}One can further observe that  
$$
H_m = \frac{m}{\beta^*} \sum_{j=1}^{m-1} \frac{(-1)^{[\frac{2j}{m}]}}{j}.
$$
Indeed,  we have the identity
$$
(-1)^m \omega_{2m-1} \frac{2m}{\beta^*} K_{m,\frac{j}{2}} K_{m,m-\frac{j}{2}-\frac{1}{2}} = 
\begin{Si}{cc}
 \frac{1}{j} &  j  \text{ even}, \\
\rule{0cm}{0.5cm} \frac{1}{2m-j-1}  &  j \text{ odd}.
\end{Si}
$$
Hence,
\[\begin{split}
\omega_{2m-1} \frac{2m}{\beta^*}  \sum_{j=1}^{m-1} (-1)^{j+m} K_{m,\frac{j}{2}} K_{m,m-\frac{j}{2}-\frac{1}{2}}  &=  \sum_{j=1,\, j \text{ even}}^{m-1} \frac{1}{j} -  \sum_{j=1,\, j \text{ odd}}^{m-1} \frac{1}{2m-j-1} \\
& = \sum_{j=1,\, j \text{ even}}^{m-1} \frac{1}{j} -  \sum_{j=m,\,  j \text{ even}}^{2m-2} \frac{1}{j} \\
& = \frac{1}{2} \sum_{j=1}^{m-1} \frac{(-1)^{[\frac{2j}{m}]}}{j}.
\end{split}
\]
\end{rem}

We conclude this section, by recalling the following standard consequence of Adams' inequality and the density of $C^{\infty}_c(\Omega)$ in $H^m_0(\Omega)$.

\begin{lemma}\label{integrability}
For any $u\in H^m_0(\Omega)$ and $\beta\in \R^+$, we have $e^{\beta u^2}\in L^1(\Omega)$. 
\end{lemma} 
\begin{proof}
For any $\eps>0$ we can find a function $v_{\eps}\in  C^\infty_0(\Omega)$ such that  $\|v_\eps - u \|_{H^m_0(\Omega)}^2\le \eps $. Since 
$$
u^2 = v_\eps^2 + (u-v_\eps)^2 + 2 v_\eps (u-v_\eps) \le 2 v_\eps^2 + 2 (u-v_\eps)^2,
$$
we have 
$$
e^{\beta u^2} \le \|e^{2\beta v_\eps^2}\|_{L^\infty(\Omega)} e^{2\beta (u-v_\eps)^2} \le  \|e^{2\beta v_\eps^2}\|_{L^\infty(\Omega)} e^{2\beta\eps  \bra{\frac{u-v_\eps}{\|u-v_\eps\|_{H^m_0(\Omega)}}}^2}.
$$
If we choose $\eps>0$ small enough, we get $2\eps \beta \le \beta^*$ and, applying Adam's inequality \eqref{Adams}, we find 
$$
\int_{\Omega} e^{\beta u^2}dx 
\le \|e^{2\beta v_\eps^2}\|_{L^\infty(\Omega)} F_{\beta^*}\bra{ \frac{u-v_\eps}{\|u-v_\eps\|_{H^m_0(\Omega)}}} <+\infty. 
$$
\end{proof}

\section{Subcritical inequalities and their extremals}\label{subcrit}
In this section, we prove the existence of extremal functions  for $F_\beta$ on $M_\alpha$ in the subcritical case  $\beta<\beta^*$, $0\le \alpha<\lambda_1(\Omega)$. As in the case $m=1$, this is a consequence of Vitali's convergence theorem and of the following improved Adams-type inequality, which is a generalization of Theorem $1.6$ in \cite{Lions}.  

\begin{prop}\label{Lions}
Let $u_n\in H^m_0(\Omega)$ be a sequence of functions such that $\|u_n\|_{H^m_0(\Omega)}\le 1$ and $u_n\rw u_0$ in $H^m_0(\Omega)$. Then, for any $0<p<\frac{1}{1-\|u_0\|_{H^m_0}^2}$, we have 
$$
\limsup_{n\to +\infty } F_{p\beta^*} (u_n) <+\infty. 
$$
\end{prop}
\begin{proof}
First, we observe that 
$$
\|u_n-u_0\|_{H^m_0(\Omega)}^2 = \|u_n\|_{H^m_0(\Omega)}^2 + \|u_0\|_{H^m_0(\Omega)}^2 - 2(u_n,u_0)_{H^m_0(\Omega)} \le 1 - \|u_0\|_{H^m_0(\Omega)}^2 + o(1). 
$$ 
Hence, there exists $\sigma>0$ such that    
$$
p\|u_n-u_0\|_{H^m_0(\Omega)}^2 \le \sigma <1, 
$$
for sufficiently large $n$. For any $\gamma>0$, we have
$$
u_n^2 
 \le (1+\gamma^2)  u_0^2  + (1+\frac{1}{\gamma^2}) (u_n-u_0)^2.
$$
Since $0<\sigma<1$, we can choose $\gamma$ sufficiently large so that $\sigma \bra{1+\frac{1}{\gamma^2}}<1$. Applying H\"older's inequality with exponents $q =\frac{1}{\sigma \bra{1+\frac{1}{\gamma^2}} } $ and $q'= \frac{q}{q-1}$, we get 
$$
F_{p \beta^*} (u_n) \le  \int_{\Omega} e^{p\beta^* (1+\gamma^2) u_0^2} e^{p\beta^* (1+\frac{1}{\gamma^2}) (u_n-u_0)^2} dx  \le \|e^{p\beta^* (1+\gamma^2) u_0^2}\|_{L^{q'}(\Omega)} \| e^{p\beta^* (1+\frac{1}{\gamma^2}) (u_n-u_0)^2} \|_{L^q(\Omega)}.
$$
Lemma \ref{integrability} guarantees that  $\|e^{p\beta^* (1+\gamma^2) u_0^2}\|_{L^{q'}(\Omega)}<+\infty$. Moreover, since 
$$
p q (1+\frac{1}{\gamma^2}) \|u_n-u_0\|_{H^m_0(\Omega)}^2 = \frac{p}{\sigma}\|u_n-u_0\|_{H^m_0(\Omega)}^2 \le 1,
$$
for large $n$, Adams' inequality \eqref{Adams} yields
$$
 \| e^{p\beta^* (1+\frac{1}{\gamma^2}) (u_n-u_0)^2} \|_{L^q(\Omega)} = F_{\beta^*} \bra{\sqrt{p q (1+\frac{1}{\gamma^2})}(u_n-u_0)}^\frac{1}{q} \le S_{0,\beta^*}^\frac{1}{q}<+\infty. 
$$
Hence, $\limsup_{n\to + \infty } F_{p\beta^*} (u_n) <+\infty.$
\end{proof}

Next we recall the following consequence of Vitali's convergence theorem (see e.g. \cite{Rudin}
).  

\begin{trm}
\label{Vitali}
Let $\Omega\subseteq \R^{2m}$ be a bounded open set and take a sequence $\{f_n\}_{n\in \N} \subseteq L^1(\Omega).$ Assume that:
\begin{enumerate}
\item For a.e. $x\in \Omega$ the pointwise limit $f(x):=\lim_{n\to + \infty}f_n(x)$ exists.
\item There exists $p>1$ such that $\|f_n\|_{L^p(\Omega)}\le C$. 
\end{enumerate}
Then, $f\in L^1(\Omega)$ and $f_n\to f$ in $L^1(\Omega)$. 
\end{trm}

We can now prove the existence of subcritical extremals. 

\begin{prop}\label{sub}
For any $\beta <\beta^*$ and $0\le \alpha <\lambda_1(\Omega)$, we have $S_{\alpha,\beta}<+\infty$. Moreover $S_{\alpha, \beta}$ is attained, i.e., there exists $u_{\alpha,\beta}\in M_{\alpha}$ such that $S_{\alpha,\beta}=F_\beta(u_{\alpha,\beta}).$ 
\end{prop}
\begin{proof}
Let $u_n \in M_\alpha$ be a maximizing sequence for $F_\beta$, i.e. such that $F_\beta(u_n)\to S_{\alpha,\beta}$ as $n\to + \infty$. Since $F_\beta(u_n)\le F_\beta(\frac{u_n}{\|u_n\|_\alpha})$, w.l.o.g we can assume $\|u_n\|_\alpha =1$, for any $n\in \N$. 
Since $\alpha<\lambda_1(\Omega)$, $u_n$ is uniformly bounded in $H^m_0(\Omega)$. In particular, extracting   a subsequence,  we can find $u_0\in H^m_0(\Omega)$ such that $u_n\rw u_0$ in $H^m_0(\Omega)$, $u_n\to u_0$ in $L^2(\Omega)$ and $u_n \to u_0$ a.e. in $\Omega$. Observe that 
\[
\begin{split}
\|u_0\|_{\alpha}^2 = \|u_0\|_{H^m_0(\Omega)}^2 -\alpha \|u_0 \|_{L^2(\Omega)}^2  \le \liminf_{n\to +\infty }\|u_n\|_{H^m_0(\Omega)}^2 -\alpha \|u_n \|_{L^2(\Omega)}^2 = \liminf_{n\to +\infty }\|u_n\|_{\alpha} ^2 =  1, 
\end{split}
\]
hence $u_0\in M_\alpha$. If we prove that there exists $p>1$ such that 
\begin{equation}\label{p>1}
\|e^{\beta u_n^2}\|_{L^p(\Omega)}\le C,
\end{equation} 
then we can apply Theorem \ref{Vitali} to $f_n:= e^{\beta u_n^2}$ and we obtain $F_{\beta}(u_0)=S_{\alpha,\beta}$ and $S_{\alpha,\beta}<+\infty$, which concludes the proof. To prove \eqref{p>1} we shall treat two differnt cases. 

Assume first that $u_0=0$. Then we have 
$$
 \beta \|u_n\|_{H^m_0(\Omega)}^2 =  \beta(1 +\alpha \|u_n  \|_{L^2(\Omega)}^2 )  = \beta + o(1) < \beta^*,
$$
and we can find $p >1$ such that 
$$
 p \beta \|u_n\|_{H^m_0(\Omega)}^2 \le \beta^*,
$$
for $n$ large enough. In particular, using \eqref{Adams}, we obtain
$$
\|e^{\beta u_n^2}\|_{L^p(\Omega)}^p  = \int_{\Omega} e^{p\beta u_n^2} dx \le F_{\beta^*}\bra{\frac{u_n}{\|u_n\|_{H^m_0(\Omega)}}}\le S_{0,\beta^*}<+\infty.  
$$

Assume instead $u_0\neq 0$. Consider the sequence $v_n:= \frac{u_n}{\|u_n\|_{H^m_0(\Omega)}}$, and observe that $v_n\rw v_0$ in $H^m_0(\Omega)$ where $v_0 = \frac{u_0}{\sqrt{1+\alpha\|u_0\|^2_{L^2}}}$.  Since
\[\begin{split}
\|u_n\|_{H^m_0}^2 (1-\|v_0\|_{H^m_0}^2)  & = \bra{1+\alpha \|u_n\|_{L^2}^2 }\bra{1 -\frac{\|u_0\|_{H^n_0}^2}{1+\alpha\|u_0\|^2_{L^2}}}  \\
& =     1 + \alpha\|u_0\|_{L^2}^2- \|u_0\|_{H^n_0}^2 + o(1)\\
& =     1 - \|u_0\|_{\alpha}^2 + o(1),
\end{split}
\]
and $u_0\neq 0$, we get
$$
\limsup_{n\to +\infty} \|u_n\|_{H^m_0}^2 < \frac{1}{1-\|v_0\|_{H^m_0}^2}. 
$$
In particular, there exist $p,q>1$ such that 
$$
p\| u_n\|_{H^m_0}^2\le q <   \frac{1}{1-\|v_0\|_{H^m_0}^2},
$$
for $n$ large enough. Then, we get 
$$
\|e^{\beta u_n^2}\|_{L^p}^p  \le \|e^{\beta^* u_n^2}\|_{L^p}^p = \|e^{\beta^* \|u_n\|_{H^m_0}^2  v_n^2}\|_{L^p}^p \le \|e^{\beta^* q v_n^2}\|_{L^1}= F_{q\beta^*}(v_n) \le C,
$$
where the last inequality follows from Proposition \ref{Lions}.  Therefore, the proof of \eqref{p>1} is complete. 
\end{proof}

Finally, we stress that, as $\beta \to \beta^*$, the family $u_{\alpha,\beta}$ is a maximizing family for the critical functional $F_{\beta^*}$.

\begin{lemma}\label{limsub}
For any $0\le \alpha <\lambda_1(\Omega)$, we have 
$$
\lim_{\beta \nearrow \beta^*} S_{\alpha,\beta}= S_{\alpha,\beta^*}. 
$$
\end{lemma}
\begin{proof}
Clearly, $S_{\alpha,\beta}$ is monotone increasing with respect to $\beta$. In particular, we must have 
$$
  \lim_{\beta\nearrow \beta^* } S_{\alpha,\beta}\le S_{\alpha,\beta^*}.
$$
To prove the opposite inequality, we observe that, for any function $u\in M_\alpha$, the monotone convergence theorem implies 
$$
F_{\beta^*}(u) = \lim_{\beta\nearrow \beta^*} F_\beta(u) \le   \lim_{\beta\nearrow \beta^*} S_{\alpha,\beta}. 
$$
Since $u$ is an arbitrary function in $M_\alpha$, we get 
$$
S_{\alpha,\beta^*} \le \lim_{\beta \nearrow \beta^*} S_{\alpha,\beta_n}.
$$

\end{proof}

%

\section{Blow-up analysis at the critical exponent}\label{main sec}
%

In this section, we will study the behaviour of subcritical extremals as  $\beta$ approaches the critical exponent  $\beta^*$ from below. In the following, we will take a sequence $(\beta_n)_{n\in \N}$ such that
\begin{equation}\label{betan}
0< \beta_n <\beta^* \quad \mbox{ and } \quad \beta_n \to \beta^*, \text{ as } n\to+\infty.
\end{equation}
Due to Proposition \ref{sub}, for any $n\in \N$, we can find a function $u_n\in M_{\alpha}$ such that 
\begin{equation}\label{extremal}
F_{\beta_n}(u_n)= S_{\alpha,\beta_n}.
\end{equation}

\begin{lemma}\label{primo}
If $u_n\in M_\alpha$ satisfies \eqref{extremal}, then $u_n$ has the following properties
\begin{enumerate}
\item $\|u_n\|_\alpha=1$.
\item $u_n$  is a  solution to 
\begin{equation}\label{star}
\begin{Si}{ll}
(-\Delta )^m u_n = \lambda_n u_n e^{\beta_n u_n^2} +\alpha u_n & \mbox{ in }\Omega, \\
u_n = \partial_\nu u_n= \cdots = \partial^{m-1}_{\nu} u_n = 0 & \mbox{ on }\partial \Omega, 
\end{Si}
\end{equation}
where 
\begin{equation}\label{lambdan}
\lambda_n = \bra{\int_{\Omega} u_n^2 e^{\beta_n u_n^2}dx }^{-1}. 
\end{equation}
\item $u_n\in C^{\infty}(\ov{\Omega})$.
\item $F_{\beta_n}(u_n) \to S_{\alpha,\beta^*}$ as $n\to +\infty$. 
\item If $\lambda_n$ is as in \eqref{lambdan}, then $\dis{\limsup_{n\to +\infty} \lambda_n <+\infty}$. 
\end{enumerate}
\begin{proof}
\emph{1.} Since $u_n\in M_{\alpha}$, we have $\|u_n\|_{\alpha}\le 1$, $\forall\; n\in \N$. Moreover, the maximality of $u_n$ implies $u_n\neq 0$. If $\|u_{n}\|_\alpha<1$, then we would have
$$
S_{\alpha,\beta_n } = F_{\beta_n}(u_n) < F_{\beta_n}\bra{\frac{u_n}{\|u_n\|_{\alpha}}},
$$
which is a contradiction. 

\emph{2.} 
Since $u_n$ is a critical point  for $F_{\beta_n}$ constrained to $M_\alpha$,  there exists $\gamma_n \in \R$ such that 
\begin{equation}\label{e1}
\gamma_n\bra{ (u_n,\ph)_{H^m_0} -\alpha (u_n,\ph)_{L^2} }  = \beta_n \int_{\Omega} u_n e^{\beta_n u_n^2}\ph dx,
\end{equation}
for any  function $\ph \in H^m_0(\Omega)$. Taking $u_n$ as test function and using \emph{1}., we find
\begin{equation}\label{e2}
\gamma_n = \beta_n \int_{\Omega} u_n^2 e^{\beta_n u_n^2} dx.
\end{equation}
In particular, $\gamma_n \neq 0$ and \eqref{e1} implies that $u_n$ is a weak solution of \eqref{star} with  $\lambda_n:=\frac{\beta_n}{\gamma_n}$. Finally, \eqref{e2} is equivalent to \eqref{lambdan}.

\emph{3.} By Lemma \ref{integrability}, we know that $u_n$ and $e^{\beta_n u_n^2}$ belong to every $L^p$ space, $p>1$. Then, applying standard elliptic regularity results (see e.g. Proposition \ref{ell zero}) and Sobolev embedding theorem, we find $u_n \in W^{2m,p}(\Omega)\subseteq C^{2m-1,\gamma}(\Omega)$, for any $\gamma\in (0,1)$. Then, we also have $(-\Delta)^m u_n\in C^{2m-1,\gamma}(\Omega)$ and, applying recursively Schauder estimates (Proposition \ref{ell shau}), we conclude that $u_n \in C^{\infty}(\ov \Omega)$.

\emph{4.} This is a direct consequence of Lemma \ref{limsub}.

\emph{5.}   Assume by contradiction that there exists a subsequence for which $\lambda_n \to +\infty$, as $n\to +\infty$. Then, by \eqref{lambdan}, we have 
$$
\int_{\Omega}u_n^2 e^{\beta_n u_n^2} dx \to 0,
$$
as $n\to +\infty$. Exploiting the basic inequality  $e^{t}\le 1+t e^{t}$ for $t\ge 0$, we obtain
$$
F_{\beta_n}(u_n) \le |\Omega|+\beta_n \int_{\Omega} u_n^2 e^{\beta_n u_n^2}dx \to |\Omega|.
$$
Since, by \emph{4.}, $F_{\beta_n}(u_n) = S_{\alpha,\beta_n}\to S_{\alpha,\beta^*} >|\Omega|$, we get a contradiction. 

\end{proof}
\end{lemma}

In order to prove that $S_{\alpha,\beta^*}$ is finite and attained, we need to show that $u_n$ does not blow-up as $n\to +\infty$.  Let us take a point $x_n \in \Omega$ such that 
\begin{equation}\label{mun and xn}
\mu_n:=\max_{\ov \Omega} |u_n| =|u_n(x_n)|.
\end{equation}
Extracting a subsequence and changing the sign of $u_n$ we can always assume that
\begin{equation}\label{mun and xn2}
u_n (x_n) = \mu_n \quad \mbox{ and } \quad x_n \to  x_0 \in \ov \Omega, \text{ as }n\to +\infty. 
\end{equation}
The main purpose of this  section consists in proving the following Proposition.

\begin{prop}\label{big}
Let $\beta_n$, $u_n$, $\mu_n$, $x_n$, and $x_0$ be as in \eqref{betan}, \eqref{extremal}, \eqref{mun and xn}, and \eqref{mun and xn2}. If $\mu_n\to +\infty$, then  $x_0\in \Omega$ and we have
$$
S_{\alpha,\beta^*} = \lim_{n\to +\infty} F_{\beta_n}(u_n) \le|\Omega| + \frac{\omega_{2m}}{2^{2m}} e^{\beta^* \bra{C_{\alpha,x_0} - I_m}},
$$ 
where $C_{\alpha,x_0}$ is as in Proposition \ref{prop green} and 
\begin{equation}\label{Imbrutto}
I_m:= - \frac{m 4^{2m}}{\beta^*\omega_{2m} }\int_{\R^{2m}} \frac{\log\bra{1+\frac{|y|^2}{4}}}{\bra{4+|y|^2}^{2m}}dy.
\end{equation}
\end{prop}

The proof of Proposition \ref{big} is quite long and it will be divided into several subsections. Some standard properties of $u_n$ will be established  in section 4.1. Then, in section 4.2, as a consequence of Lorentz-Zygmund elliptic estimates, we will prove uniform bounds for $\Delta u_n^2$. Such bounds will be crucial in the analysis given in section 4.3, where we will study the behaviour of $u_n$ on a small scale. Sections 4.4, 4.5 and 4.6 contain respectively estimates on the derivatives of $u_n$, the definition of suitable polyharmonic truncations of $u_n$, and the description of the behaviour of $u_n$ far from $x_0$. In section 4.7 we will deal with blow-up at the boundary, which will be excluded using Pohozaev-type identities. Finally, we conclude the proof in section 4.8 by introducing a new technique to obtain lower bounds on the Dirichlet energy for  $u_n$ on suitable annular regions. 

In the rest of this section $\beta_n$, $u_n$, $\mu_n$, $x_n$, and $x_0$ will always be as in Proposition \ref{big} and we will always assume that $\mu_n \to  +\infty$. 

\subsection{Concentration near the blow-up point}
In this subsection we will prove that, if $\mu_n \to +\infty$, $u_n$ must concentrate around the blow-up point $x_0$. We start by proving that its weak limit in $H^m_0(\Omega)$ is $0$.   

\begin{lemma}\label{weak lim}
If $\mu_n \to +\infty$, then $u_n \rw 0$  in $H^m_0(\Omega)$ and  $u_n \to 0$ in $L^p(\Omega)$ for any $p\ge 1$. 
\end{lemma}
\begin{proof}
 Since $u_n$ is bounded in $H^m_0(\Omega)$, we can assume that $u_n \rw u_0$ in $H^m_0(\Omega)$ with $u_0 \in H^m_0(\Omega)$.  The compactness of the embedding of $H^m_0(\Omega)$ into $L^p(\Omega)$ implies $u_n \to u_0$ in $L^p(\Omega)$, for any $p\ge 1$.  
 
 If $u_0\neq 0$, then, by Proposition \ref{Lions},  $e^{\beta_n u_n^2}$ is bounded in $L^{p_0}(\Omega)$ for some $p_0>1$. By Lemma \ref{primo}, we know that $\lambda_n$ is bounded. Hence $\lm u_n$ is bounded in $L^{s}(\Omega)$ for any $1<s<p_0$. Then, by elliptic estimates (see Proposition \ref{ell zero}), we find that $u_n$ is bounded in $W^{2m,s}(\Omega)$ and, by Sobolev embeddings, in $L^\infty(\Omega)$. This contradicts $\mu_n \to +\infty$. Hence, we have $u_0=0$. 

\end{proof}

In fact, $u_n$ converges to $0$ in a much stronger sense if we stay far from the blow-up point $x_0$, while $|\Delta^\frac{m}{2} u_n|^2 $ concentrates around $x_0$.

\begin{lemma}\label{conv0C} If $\mu_n\to +\infty$, then we have:
\begin{enumerate}
\item $|\Delta^\frac{m}{2} u_n|^2 \rw \delta_{x_0}$ in the sense of measures.
\item $e^{\beta_n u_n^2}$ is bounded in $L^s(\Omega\setminus B_\delta(x_0))$, for any $s\ge1$, $\delta>0$. 
\item $u_n\to 0$ in $C^{2m-1,\gamma}(\Omega\setminus B_\delta(x_0))$, for any $\gamma\in (0,1)$, $\delta>0$.
\end{enumerate}
\end{lemma} 
\begin{proof}
First of all, for any function $\xi\in C^{2m}(\ov{\Omega})$,  we observe that
$$
\Delta^{\frac{m}{2}} (u_n \xi)  =  \xi \Delta^\frac{m}{2} u_n  +  f_n, 
$$
with 
$$
|f_n|\le C_1 \sum_{l=0}^{m-1} |\nabla^l u_n||\nabla^{m-l} \xi| \le C_2   \sum_{l=0}^{m-1} |\nabla^l u_n|,
$$
for some constants $C_1,C_2>0$, depending only on $m,l,$ and $\xi$. Since $u_n\rw 0$ in $H^m_0(\Omega)$, and $H^m_0(\Omega)$ is compactly embedded in $H^{m-1}(\Omega)$, we get that $f_n \to 0$ in $L^2(\Omega)$. In particular, we have 
\begin{equation}\label{cutoff1}
\begin{split}
\|\Delta^\frac{m}{2} ( u_n \xi)\|^2_{L^2(\Omega)} &= \int_{\Omega} \xi^2 |\Delta^\frac{m}{2} u_n |^2dx + 2 \int_{\Omega} \Delta^\frac{m}{2} u_n \cdot f_n dx  + \int_{\Omega} |f_n|^2 dx \\
& = \int_{\Omega} \xi^2 |\Delta^\frac{m}{2} u_n |^2dx +o(1).
\end{split}
\end{equation}

We can now prove the first statement of this lemma. Assume by contradiction that there exists $r>0$ such that 
\begin{equation}\label{cut2}
\limsup_{n\to +\infty} \|\Delta^\frac{m}{2} u_n\|_{L^2(B_{r}(x_0)\cap \Omega)}^2 <1.
\end{equation}
Take a function $\xi \in C^\infty_c(\R^{2m})$ such that $\xi \equiv 1$ on $B_\frac{r}{2}(x_0)$, $\xi \equiv 0$ on $\R^{2m}\setminus B_{r}(x_0)$ and $0\le \xi\le 1$. By \eqref{cutoff1} and \eqref{cut2}, we have that $
\limsup_{n\to + \infty} \|\Delta^\frac{m}{2} ( u_n\xi) \|_{L^2(\Omega)}^2 <1.
$
Adams' inequality implies that we can find $s> 1$ such that $e^{\beta_n (u_n\xi)^2}$ is bounded in  $L^s(\Omega)$. In particular, $e^{\beta_n u_n^2}$ is bounded in $L^{s}(B_\frac{r}{2}(x_0))$. By Lemma \ref{weak lim}, $u_n \to 0$ in $L^p(\Omega)$ for any $p\ge 1$. Therefore, we get that $\lm u_n \to 0 $ in  $L^{q}(\Omega)$ for any $1<q<s$.   Then, Proposition \ref{ell zero} yields  $u_n \to 0$  in $W^{2m,q}(\Omega)$ and, since $q>1$, in $L^\infty(\Omega)$. This contradicts $\mu_n\to +\infty$. 

To prove \emph{2.}, we fix a cut-off function $\xi_2\in C^{\infty}_c(\R^{2m})$ such that $\xi_2 \equiv 1$ in $\R^{2m}\setminus B_\delta(x_0)$, $\xi_2 \equiv 0$ in $B_\frac{\delta}{2}(\Omega)$, and $\xi \le 1$. Since $|\Delta^\frac{m}{2} u_n |\rw \delta_{x_0}$, from \eqref{cutoff1} we get $\|\Delta^\frac{m}{2} (u_n\xi_2) \|_{L^2(\Omega)} \to 0$. Then, Adams' inequality implies that $e^{\beta_n (u_n \xi_2)^2}$ is bounded in $L^s(\Omega)$, for any $s>1$. Because of the definition of $\xi_2$, we get the conclusion.  

To prove \emph{3.}, we apply standard elliptic estimates. By part \emph{2.}, we know that $u_n$ and $e^{\beta_n u_n^2}$ are bounded in $L^s(\Omega \setminus B_\delta(x_0))$ for any $s\ge 1$.  Since $\lambda_n$ is bounded, the same bound holds for $\lm u_n$. Then, elliptic estimates (Propostion \ref{ell new}) imply that $u_n $ is bounded in $W^{2m,s}(\Omega\setminus B_{2\delta}(x_0))$.  By Sobolev embedding theorem, it is also bounded in $C^{2m-1,\gamma}(\Omega \setminus B_{2\delta}(x_0))$, for any $\gamma \in (0,1)$. Then, up to a subsequence, we can find a function $u_0\in C^{2m-1,\gamma}(\Omega \setminus B_{2\delta}(x_0))$ such that $u_n  \to u_0$ in $C^{2m-1,\gamma}(\Omega \setminus B_{2\delta} (x_0))$. Since $u_n \rw 0$ in $H^m_0(\Omega)$, we must have $u_0 \equiv 0$ in $\Omega \setminus B_{2\delta}(x_0)$ and $u_n \to 0$ in $C^{2m-1,\gamma}(\Omega \setminus B_{2\delta}(x_0))$.        
\end{proof}

\subsection{Lorentz-Sobolev elliptic estimates}
In this subsection, we prove uniform integral estimates on the derivatives of $u_n$.  Notice that Sobolev's inequality implies $\|\nabla^l u_n\|_{L^{\frac{2m}{l}}(\Omega)}\le C$ for any $1\le l \le m-1$. In addition, standard elliptic estimates (Proposition \ref{ell L1}) yield  $\|\nabla^l u_n\|_{L^p(\Omega)}\le C$, for any $p<\frac{2m}{l}$ and $m \le l\le 2m-1$. Arguing as in \cite{mar}, we will prove that sharper estimates can be obtained  thanks to Lorentz-Zygmund elliptic regularity theory (see Proposition \ref{ell Lor} in Appendix). In the following, for any $\alpha\ge 0$, $1<p<+\infty$, and $1\le q \le +\infty$,  $(L(\log L)^\alpha,\|\cdot\|_{L(\log L)^\alpha})$ and $(L^{(p,q)}(\Omega),\|\cdot\|_{(p,q)})$, will denote respectively the Zygmund and Lorentz spaces on $\Omega$.   We refer to the Appendix for the  precise definitions (see \eqref{zygm}-\eqref{Lor fin}).  

\begin{lemma}\label{Lorentz} For any $1\le l\le 2m-1$, we have
$$\| \nabla^l u_n \|_{(\frac{2m}{l},2)} \le C.$$ 
\end{lemma}
\begin{proof}
Set $f_n:= \lm u_n$. By Proposition \ref{ell Lor}, there exists a constant $C>0$ such that 
$$
\|\nabla^{l} u_n\|_{(\frac{2m}{l},2)}\le C \|f_n\|_{L(Log L)^\frac{1}{2}}, 
$$
for any $1\le l\le 2m-1$, $n\in \N$. Therefore, it is sufficient to prove that $f_n$ is bounded in $L(\log L)^\frac{1}{2}$.For any $x \in \R^+$, let $\log^+x := \max\{0,\log x\}$ be the positive part of $\log x$.  Since $\beta_n$ and $\lambda_n$ are bounded, using the simple inequalities 
$$
\log(x+y) \le x+ \log^+ y  \qquad \text{and} \qquad \log^+(xy) \le \log^+x  +\log^+ y,\qquad x,y\in \R^+,
$$
we find
\[
\begin{split}
\log(2+|f_n|) &  \le   2 + \log^+|u_n| + \log^+ \bra{ \lambda_n e^{\beta_n u_n^2}+\alpha}\\
 & \le C+  \log^+ |u_n| +   \beta_n u_n^2 \\
  & \le C (|u_n| + 1)^2.
\end{split}
\]
Then,
$$
|f_n| \log^\frac{1}{2}(2+|f_n|) \le C |f_n|(1+|u_n|) \le C  \bra{ \lambda_n |u_n| e^{\beta_n u_n^2} +    \lambda_n u_n^2 e^{\beta_n u_n^2} + \alpha |u_n| + \alpha u_n^2 },
$$
and, by Lemma \ref{weak lim} and \eqref{lambdan}, as $n\to + \infty$ we get
\[\begin{split}
\int_{\Omega } |f_n| \log^\frac{1}{2}(2+|f_n|)dx &\le C \bra{\lambda_n \int_{\Omega} |u_n| e^{\beta_n u_n^2} dx + 1+ o(1)} \\
 & \le  C \bra{\lambda_n \int_{\{ |u_n|<1\}} |u_n| e^{\beta_n u_n^2} dx + \lambda_n \int_{\{ |u_n|\ge 1\}} |u_n|^2 e^{\beta_n u_n^2} dx  +  1+ o(1)}\\
 & \le C\bra{\lambda_n  e^{\beta_n} |\Omega|+2 + o(1)} =O(1).
\end{split}
\]
Hence, $f_n$ is bounded in $L(Log L)^\frac{1}{2}$. 
\end{proof}

As a consequence of Lemma \ref{Lorentz}, we obtain an integral estimate on the derivatives of $u_n^2$,   which will play an important role in Sections 4.3 and 4.4. The idea behind this estimate is based on the following  remark: up to terms involving only lower order derivatives, which can be controlled using Lemma \ref{Lorentz}, $\lm u_n^2$ coincides with   $u_n\lm u_n$, which  is bounded in $L^1(\Omega)$. Then, estimates on $u_n^2$  can be obtained via Green's representation formula.

\begin{lemma}\label{integral est}
There exists a constant $C>0$ such that  for any $1\le l\le 2m-1$, $x\in \Omega$,  and $\rho>0$ with $B_{\rho}(x)\subseteq \Omega$, we have
$$
\int_{B_{\rho}(x)} |\nabla^l u_n^2| dy \le C \rho^{2m-l}.
$$ 
\end{lemma}
\begin{proof}
We start by observing that $\lm u_n^2$ is bounded in $L^1(\Omega)$.  Clearly 
\[
\begin{split}
|\lm u_n^2 | &\le  2 |u_n\lm u_n| + C \sum_{j=1}^{2m-1}|\nabla^j u_n| |\nabla^{2m-j} u_n|.
\end{split}
\]
Equation \eqref{lambdan} and Lemma \ref{weak lim} imply that $u_n\lm u_n$ is bounded in $L^1(\Omega)$. As a consequence of  H\"older's inequality for  Lorentz spaces (Proposition \eqref{HolLor}) and Lemma \ref{Lorentz}, we find  
$$\int_{\Omega} |\nabla^{2m-j} u_n||\nabla^j u_n|dx \le \|\nabla^{2m-j} u_n \|_{(\frac{2m}{2m-j},2)}\|\nabla^{j} u_n \|_{(\frac{2m}{j},2)} \le C.
$$ Thus, $\lm u_n^2$ is bounded in $L^1(\Omega)$. 

Now, we apply  Green's representation formula to $u_n^2$ to get 
$$
u_n^2 (y) = \int_{\Omega} G_y (z) \lm u_n^2(z) dz,
$$
for any $y \in \Omega$
where $G_y:= G_{0,y}$ is defined as in \eqref{green}. By the properties of $G_y$ (see Proposition \ref{prop green}), we have 
$$
|\nabla^l_y G_y (z) |\le \frac{C}{|y-z|^l},
$$ for any $y,z\in \Omega$ with $z\neq y$. Hence
$$
|\nabla^l u_n^2(y)|\le \int_{\Omega} \frac{C |\lm u_n^2(z)|}{|y-z|^l} dz. 
$$
Let $x\in \Omega$ and   $\rho>0$  be as in the statement. Then, we find
\[
\begin{split}
\int_{B_\rho(x)} |\nabla^l u_n^2| dy &\le \int_{B_\rho(x)} \int_{\Omega} \frac{C |\lm u_n^2(z)|}{|y-z|^l} dzdy \\
&=  C \int_{\Omega} |\lm u_n^2(z)| \int_{B_\rho(x)} \frac{1}{|y-z|^l} dydz.
\end{split}
\]
Since 
$$
 \int_{B_\rho(x)} \frac{1}{|y-z|^l} dy \le  \int_{B_\rho(x)} \frac{1}{|y-x|^l} dy = C \rho^{2m-l},
$$
and $\lm u_n^2$ is bounded in $L^1(\Omega)$, we get the conclusion. 
\end{proof}

\subsection{The behavior on a small scale}
Let $u_n$, $\mu_n$ and $x_n$ be as in \eqref{extremal}, \eqref{mun and xn}, \eqref{mun and xn2}. In this  subsection, we  will study the behavior of $u_n$ on small balls centered at the maximum point $x_n$. Define $r_n>0 $ so that 
\begin{equation}\label{rn}
\omega_{2m} r_n^{2m} \lambda_n \mu_n^2 e^{\beta_n \mu_n^2}=1,
\end{equation}
with $\omega_{2m}$ as in \eqref{omega}.

\begin{rem}
Note that, as $n\to +\infty$, we have $r_n^{2m} = o(\mu_n^{-2})$ and, in particular, $r_n \to 0$.
\end{rem}
\begin{proof}
Indeed, by \eqref{lambdan}, we have
$$\frac{1}{\lambda_ne^{\beta_n\mu_n^2}} =  \frac{1}{e^{\beta_n \mu_n^2}}\int_{\Omega} u_n^2 e^{\beta_n u_n^2} dx \le  \|u_n\|_{L^2(\Omega)}^2.$$
Since $u_n \to 0$ in $L^2(\Omega)$, the definition of $r_n^{2m}$ yields $r_n^{2m} \mu_n^2 \to 0$ as $n\to + \infty$. 
\end{proof}

%

Let us now consider the  scaled function 
\begin{equation}\label{etan}
\eta_n (y):= \mu_n ( u_n(x_n+ r_n y) -\mu_n),
\end{equation} 
which is defined on the set
\begin{equation*}
\Omega_n :=\{y\in \R^{2m}\;:\; x_n + r_n y \in \Omega\}. 
\end{equation*}
The main purpose of this subsection consists in proving the following convergence result. 
\begin{prop}\label{conv eta}
We have $\frac{d(x_n, \partial \Omega)}{r_n} \to  +\infty$ and, in particular, $\Omega_n$ approaches $\R^{2m}$ as $n\to + \infty$. Moreover,
$\eta_n$ converges to the limit function 
\begin{equation}\label{eta0}
\eta_0 (y) = - \frac{m}{\beta^*} \log\bra{1+\frac{|y|^2}{4}}
\end{equation}
in $C^{2m-1,\gamma}_{loc}(\R^{2m})$, for any $\gamma\in (0,1)$. 
\end{prop}

In order to avoid repetitions, it is convenient to see Proposition \ref{conv eta}  as  a special case of the following more general result, which will be  useful also in the proof of Proposition \ref{stima Luca}.  

\begin{prop}\label{conv eta gen}
Given two sequences $\tilde x_n\in \Omega$ and $s_n\in\R^+$,  consider the scaled set $\tilde \Omega_n :=\{ y\in \R^{2m}\;:\; \tilde x_n +s_n y \in \Omega \}$ and the functions $v_n(y):= u_n(\tilde x_n +s_n y )$ and 
$\tilde \eta_n(y):= \tilde  \mu_n \bra{ v_n (y )- \tilde \mu_n}  $, where $\tilde \mu_n := u_n(\tilde x_n)$. Assume that
\begin{enumerate}
\item $\omega_{2m} s_n^{2m} \lambda_n \tilde \mu_n^2 e^{\beta_n \tilde \mu_n^2} = 1$ and $|\tilde \mu_n| \to +\infty$, $s_n^{2m} \to 0$,   as $n\to+ \infty$.
\item For any $R>0$ there exists a constant $C(R)>0$ such that  
\begin{equation}\label{ass}
\left|\frac{v_n}{\tilde \mu_n} \right| \le C(R) \quad \mbox{and }\quad v_n^2 - \tilde \mu_n^2 \le  C(R) \quad \text{ in }\tilde \Omega_n \cap B_{R}(0).
\end{equation}
\end{enumerate}
Then, we have $\frac{d(\tilde x_n, \partial \Omega)}{s_n} \to  +\infty$ and $\frac{v_n}{\mu_n}\to 1$ in $C^{2m-1,\gamma}_{loc}(\R^{2m})$, for any $\gamma\in (0,1)$.   Moreover  $\tilde \eta_n\to \eta_0$  in $C^{2m-1,\gamma}_{loc}(\R^{2m})$, where $\eta_0$ is defined as in \eqref{eta0}.
\end{prop}

Note that the assumptions of Proposition \ref{conv eta gen} are satisfied when $\tilde x_n =x_n$ and $s_n =r_n$. Hence, Proposition \ref{conv eta} follows from Proposition \ref{conv eta gen}. We split the proof of Proposition \ref{conv eta gen} into four steps. The first two steps (Lemma \ref{boundarydist0} and Lemma \ref{boundarydist}) are stated under more general assumptions, since they will be reused in the proof of Proposition \ref{stime grad}.

\begin{lemma}\label{boundarydist0}
Given two sequences $\tilde x_n \in \Omega$ and $s_n\in \R^+$,  let $\tilde \Omega_n$ and $v_n$ be defined as in Proposition \ref{conv eta gen}. Let also $\Sigma$ be a finite (possibly empty) subset of $\R^{2m}\setminus \{0\}$. Assume that 
\begin{enumerate}
\item $s_n\to 0$ and $\dis{D_n:= \max_{0\le i\le 2m-1}|\nabla^i v_n(0)|\to +\infty}$ as $n\to +\infty$.
\item For any $R>0$, there exist $C(R)>0$ and $N(R)\in \N$ such that 
$$
|v_n(y)|\le C(R)D_n \qquad \mbox{ and } \qquad |\lm v_n(y)| \le C(R)D_n,
$$
for any $\dis{y\in \tilde \Omega_{n,R}:= \tilde \Omega_n \cap B_R(0)\setminus \bigcup_{\xi \in \Sigma} B_\frac{1}{R}(\xi)}$ and any $n\ge N(R)$. 
\end{enumerate} 
Then, we have  $$
\lim_{n\to + \infty} \frac{d(\tilde x_n,\partial \Omega)}{s_n} = +\infty.
$$
\end{lemma}
\begin{proof}
Let us consider the functions $w_n(y):= \frac{v_n(y)}{D_n}.$ First, we observe that the assumptions on $\tilde x_n$ and $s_n$ imply  
\begin{equation}\label{cond wn 1}
w_n = O(1),
\end{equation}
and 
\begin{equation}\label{cond wn 2}
|\lm w_n| = O(1),
\end{equation}
uniformly in $\tilde \Omega_{n,R}$, for any $R>0$. Moreover, by Sobolev's inequality, for any $1\le j\le m$ we have that 
\begin{equation}\label{cond wn 3}
\|\nabla^j w_{n}\|_{L^{\frac{2m}{j}}(\tilde \Omega_n)} = D_n^{-1} \|\nabla^j u_{n}\|_{L^{\frac{2m}{j}}( \Omega)} \le C D_{n}^{-1} \|\Delta^\frac{m}{2} u_n\|_{L^2(\Omega)} = O(D_{n}^{-1}).  
\end{equation}
Then, using H\"older's inequality, \eqref{cond wn 1} and \eqref{cond wn 3} give 
\begin{equation}\label{cond wn 4}
\|w_n\|_{W^{m,1}(\tilde \Omega_{n,R})} =O(1). 
\end{equation}
Now, we assume by contradiction that for a subsequence
\[
\frac{d(\tilde x_n,\partial \Omega)}{s_n} \to R_0 \in [0,+\infty).
\] 
Then, the sets $\tilde \Omega_n$ converge in $C^{\infty}_{loc}$ to a hyperplane $\mathcal P$ such that $d(0,\partial \mathcal P)=R_0$.  For any sufficiently large $R>0$ and any $p>1$, using \eqref{cond wn 2}, \eqref{cond wn 4}, Proposition \ref{ell new}, and Remark \ref{remdomains},  we find a constant $C=C(R)$ such that
$\|w_n\|_{W^{2m,p}(\tilde \Omega_{n,\frac{R}{2}})} \le C.$  Then, Sobolev's embeddings imply that $\|w_n\|_{C^{2m-1,\gamma}(\tilde \Omega_{n,\frac{R}{2}})}\le C$, for any $\gamma\in (0,1)$.  Reproducing the standard proof of the  Ascoli-Arzelà theorem, we find a function $w_0\in C^{2m-1,\gamma}_{loc}(\ov{\mathcal P}\setminus \Sigma)$ such that, up to a subsequence, we have 
\begin{equation}\label{conv wn}
w_n \to w_0 \qquad \text{ in } C^{2m-1}_{loc}(\mathcal P \setminus \Sigma)
\end{equation}
and 
\begin{equation}\label{conv wn bound}
\nabla^j w_n (\xi_n) \to \nabla^j w_0 (\xi), \quad 0\le j\le 2m-1,
\end{equation}
for any $\xi \in \ov{\mathcal P}\setminus \Sigma$ and any sequence $\{\xi_n\}_{n\in \N}$ such that $\xi_n \to \xi$. Since $w_n =0$ on $\partial \tilde \Omega_n$ and $\tilde \Omega_n$ converges to $\mathcal P$, \eqref{conv wn bound} yields $w_0 \equiv 0$ in $\partial \mathcal P \setminus \Sigma$. Furthermore, \eqref{cond wn 3} and \eqref{conv wn} imply that $\nabla w_0  \equiv 0$ in $\mathcal P\setminus \Sigma $. Therefore, $w_0 \equiv 0$ on $\ov{\mathcal P} \setminus \Sigma$.  But, by definition of $D_n$ and $w_n$, we have 
$$
\max_{0\le i\le 2m-1}|\nabla^i w_n(0)| = 1,
$$
which contradicts either \eqref{conv wn} (if $R_0>0$) or \eqref{conv wn bound} (if $R_0=0$). 
\end{proof}

\begin{lemma}\label{boundarydist}
Let $s_n$, $\tilde x_n$, $v_n$, $\tilde \Omega_n$, $D_n$ and $\Sigma$ be as in Lemma \ref{boundarydist0}. Then, $|v_n(0)|\to + \infty$ and 
$$
\frac{v_n}{v_n(0)} \to 1 \qquad \text{ in } C^{2m-1,\gamma}_{loc}(\R^{2m}\setminus \Sigma),
$$ for any  $\gamma\in (0,1)$. 
\end{lemma}
\begin{proof}
Consider the function 
$
w_n (y) := \dfrac{v_n(y)}{D_n}
$, $y\in \tilde \Omega_n$. 
As in \eqref{cond wn 1}, \eqref{cond wn 2} and \eqref{cond wn 3}, we have
\begin{equation}\label{eq risc1}
\begin{split}
w_n = O(1) \qquad \text{ and }\qquad (-\Delta)^m w_n   =  O(1),
\end{split}
\end{equation}
uniformly in $B_{R}(0)\setminus \bigcup_{\xi\in \Sigma}B_\frac{1}{R}(\xi)$, for any $R>0$, and  
\begin{equation}\label{grad zero1}
\|\nabla w_n\|_{L^{2m}(\tilde \Omega_n)}\to 0.
\end{equation}
By \eqref{eq risc1}, Proposition \ref{ell loc}, Sobolev's embeddings, and \eqref{grad zero1}, a subsequence of  $w_n$ must converge to a constant function $w_0$ in $C^{2m-1,\gamma}_{loc}(\R^{2m}\setminus \Sigma)$, for any $\gamma \in (0,1)$. In particular, we have $|\nabla^j w_n(0)|\to 0$ for any $1\le j\le 2m-1$. Then, the definitions of $D_n$ and $w_n$ give 
$$
1 = \max_{0\le j\le 2m-1} |\nabla^j w_n(0)| = |w_n(0)|,
$$
which implies that $|v_n(0)|=D_n\to +\infty$ and that $|w_0|\equiv 1$ in $\R^{2m}\setminus \Sigma$. Hence,
$$
\frac{v_n}{v_n(0)} = \frac{w_n}{w_n(0)} \to 1 \qquad \mbox{ in }C^{2m-1,\gamma}_{loc}(\R^{2m}\setminus \Sigma). 
$$
\end{proof}

Next, we let $\tilde x_n, s_n$, $\tilde \mu_n$ and $\tilde \eta_n$ be as in Proposition \ref{conv eta gen} and we  apply Lemma \ref{integral est} to prove bounds for $\Delta \tilde \eta_n$ in $L^1_{loc}(\R^{2m})$.

\begin{lemma}\label{lap etan}
 Under the assumptions of Proposition \ref{conv eta gen}, there exists a constant $C>0$ such that  
 $$
 \|\Delta \tilde  \eta_n\|_{L^1(B_R(0))} \le C R^{2m-2},
 $$
for any $R>1$ and for sufficiently large $n$.  
\end{lemma}
\begin{proof} First, we observe that $\tilde x_n$ and $s_n$ satisfy the assumptions of Lemma \ref{boundarydist0} and Lemma \ref{boundarydist}. Indeed, equation \eqref{star}, the definition of $v_n$, and the assumptions on $\tilde x_n$ and $s_n$ yield $v_n = O(|\tilde \mu_n|)$ and
\begin{equation}\begin{split}\label{verify}
\lm v_n & =  s_n^{2m } \lambda_n v_n e^{\beta_n v_n^2} +s_n^{2m}\alpha v_n \\
 & = \omega_{2m}^{-1} \frac{v_n}{\tilde \mu_n^2} e^{\beta_n (v_n^2-\tilde \mu_n^2)} + s_n^{2m}\alpha v_n \\
 & = O(|\tilde \mu_n^{-1}|) + O(s_n^{2m}|\tilde \mu_n|),
\end{split}
\end{equation}
uniformly in $\tilde \Omega_n \cap B_R(0)$, for any $R>0$. Then, Lemma \ref{boundarydist0} and Lemma \ref{boundarydist} imply that $\tilde \Omega_n$ approaches $\R^{2m}$ and 
\begin{equation}\label{rip}
\frac{v_n}{\tilde \mu_n} \to 1 \qquad \text{ in } C^{2m-1,\gamma}_{loc}(\R^{2m}), \text{ for any }\gamma\in (0,1). 
\end{equation} 
Next, we rewrite the estimates of Lemma \ref{integral est} in terms of $\tilde \eta_n$. On the one hand, by Lemma \ref{integral est}, there exists $C>0$, such that 
\begin{equation*}
\|\Delta u_n^2\|_{L^1(B_{Rs_n}(\tilde x_n))} \le C (s_n R)^{2m-2}, 
\end{equation*}
for any $R>0$ and $n\in \N$.
On the other hand, we have 
\[
\begin{split}
\|\Delta u_n^2\|_{L^1(B_{Rs_n}(\tilde x_n))}  & \ge 2 \|u_n \Delta u_n\|_{L^1(B_{Rs_n}(\tilde x_n))} -  2 \| \nabla u_n\|^2_{L^2(B_{Rs_n}(\tilde x_n))}  \\
& = 2 s_n^{2m-2} \bra{ \|v_n \Delta v_n\|_{L^1(B_{R}(0))} -  \| \nabla v_n\|^2_{L^2(B_{R}(0))} }.
\end{split}
\]
Then, we obtain 
\begin{equation}\label{stima20b}
\|v_n \Delta v_n\|_{L^1(B_{R}(0))} \le C R^{2m-2} +  \| \nabla v_n\|^2_{L^2(B_{R}(0))}.
\end{equation}
By \eqref{rip} and the definition of $\tilde \eta_n$, we infer 
\begin{equation}\label{stima20c}
\begin{split}
\|v_n \Delta v_n\|_{L^1(B_{R}(0))}  = |\tilde \mu_n | \|\Delta v_n\|_{L^1(B_{R}(0))} (1+o(1)) &=    \| \Delta \tilde \eta_n \|_{L^1(B_{R}(0))} (1+o(1))  \\ &\ge \frac{1}{2} \| \Delta \tilde \eta_n \|_{L^1(B_{R}(0))},
\end{split}
\end{equation}
for sufficiently large $n$. Finally, applying H\"older's inequality,
\begin{equation}\label{stima20d}
\| \nabla v_n\|^2_{L^2(B_{R}(0))} \le  \| \nabla v_n\|^2_{L^{2m}(B_{R}(0))}  |B_R|^{1-\frac{1}{m}} \le \| \nabla u_n\|^2_{L^{2m}(\Omega)}  |B_R|^{1-\frac{1}{m}} \le C R^{1-\frac{1}{m}}.
\end{equation}
Since $1-\frac{1}{m}\le 2m -2$, the conclusion follows from \eqref{stima20b}, \eqref{stima20c}, and \eqref{stima20d}. 
\end{proof}

We can now complete the proof of Proposition \ref{conv eta gen}. 

\begin{proof}[Proof of Proposition \ref{conv eta gen}]  Arguing as in the previous Lemma, we have that $\frac{d(\tilde x_n,\partial \Omega)}{s_n}\to +\infty$ and that \eqref{rip} holds.  Observe that \eqref{rip} implies 
\begin{equation}\label{explanation}
(1+o(1))s_n^{2m} \tilde \mu_n^2 =  \frac{s_n^{2m}}{\omega_{2m}}\int_{B_1(0) }v_n^2(y) dy = \frac{1}{\omega_{2m}} \int_{B_{s_n}(\tilde x_n)} u_n^2 (x)dx = O(\|u_n\|_{L^2(\Omega)}^2) =o(1). 
\end{equation}
Moreover, as in \eqref{verify}, by the definitions of $\tilde \eta_n$ and $v_n$, and the assumptions on $\tilde \mu_n$, $s_n$ and $\tilde x_n$,
we get 
\begin{equation}\label{lap bound}
\lm \tilde  \eta_n =  O(1) + O(s_n^{2m}\tilde \mu_n^2) =O(1),
\end{equation}
uniformly in $B_R(0)$, for any $R>0$.   In addition, Lemma \ref{lap etan} implies that $\Delta \tilde \eta_n$  is bounded in $L^1_{loc}(\R^{2m})$. By Proposition \ref{ell loc} and Sobolev's embedding theorem, $\Delta \tilde \eta_n$ is bounded in $L^\infty_{loc}(\R^{2m})$.  As a consequence of \eqref{ass} and  \eqref{rip},  we have 
$$
C(R) \ge v_n^2 - \tilde \mu_n^2 = (v_n -\tilde \mu_n)(v_n  + \tilde \mu_n) = \tilde  \eta_n (2+o(1))
$$ 
in $B_{R}(0)$. Since  $\tilde \eta_n(0)=0$, Proposition \ref{ell use} shows that $\tilde \eta_n$ is bounded in $L^\infty_{loc}(\R^{2m})$. Together with \eqref{lap bound}, Proposition \ref{ell loc}, and Sobolev's embeddings, this implies that $\eta_n$ it is bounded in $C^{2m-1,\gamma}_{loc}(\R^{2m})$, for any $\gamma\in (0,1)$. Then, we can extract a subsequence such that $\tilde \eta_n$ converges in $C^{2m-1,\gamma}_{loc}(\R^{2m})$ to a limit function $\eta_0\in C^{2m-1,\gamma}_{loc}(\R^{2m})$.  Observe that, as $n\to +\infty$,
$$
\lm \tilde \eta_n = 
 \bra{1+ \frac{\tilde\eta_n}{\tilde \mu_n^2}} \bra{ \omega_{2m}^{-1} e^{2 \beta_n \tilde \eta_n +\beta_n \frac{\tilde \eta_n^2}{\tilde \mu_n^2}} + \alpha s_n^{2m} \tilde \mu_n^2 } \to \omega_{2m}^{-1} e^{2\beta^* \eta_0},
$$ 
locally uniformly in $\R^{2m}$. This implies that $\eta_0$ must be  a weak solution of   
\begin{equation}\label{Liouville}
\begin{Si}{l}
\lm \eta_0 = \omega_{2m}^{-1} e^{2\beta^* \eta_0}, \\
e^{2\beta^*\eta_0}\in L^1(\R^{2m}), \\
\eta_0\le 0, \eta_0 (0) =0.
\end{Si}
\end{equation}  
Solutions of problem \eqref{Liouville} have been classified in \cite{MarClass} (see also \cite{Lin} and \cite{Xu}). In particular, Theorems 1 and 2 in \cite{MarClass} imply that there exists a real number $a\le 0$, such that $\lim_{|y|\to +\infty} \Delta \eta_0(y) = a$.  Moreover, either $a\neq 0$, or $\eta_0 (y) = - \frac{m}{\beta^*} \log\bra{1+\frac{|y|^2}{4}},$ for any $y\in \R^{2m}$.  To exclude the first possibility we observe that, if $a\neq 0$, then we can find $R_0>0$ such that $|\Delta \eta_0|\ge \frac{|a|}{2}$ for $|y|\ge R_0$. This yields 
\begin{equation}\label{lap eta0}
\int_{B_R(0)} |\Delta \eta_0| dy  \ge  \int_{B_{R_0}(0)} |\Delta \eta_0| dy +  \frac{|a|}{2} \omega_{2m} (R^{2m}-R_0^{2m}),
\end{equation}
for any $R>R_0$. But Lemma \ref{lap etan} implies 
\begin{equation}\label{cont1}
\int_{B_R(0)}|\Delta \eta_0| dy \le C R^{2m-2},
\end{equation}
for any $R>1$. For large values of $R$, \eqref{cont1}  contradicts \eqref{lap eta0}.
\end{proof}
 
This completes the proof of Proposition \ref{conv eta gen}. Now, we state some properties of the function $\eta_0$ that will play a crucial role in the next sections.

\begin{lemma}\label{int eta0}Let $\eta_0$ be as in \eqref{eta0}. Then, as $R\to +\infty $, we have
\begin{equation}\label{propeta1}
\omega_{2m}^{-1} \int_{B_R(0)} e^{2\beta^*\eta_0} dy=1 + O(R^{-2m})
\end{equation}
and
\begin{equation}\label{propeta2}
\int_{B_R(0)}|\Delta^\frac{m}{2}\eta_0|^2 dy= \frac{2m}{\beta^*}\log \frac{R}{2}  +  I_m   -H_m + O(R^{-2}\log R ),
\end{equation}
where $H_m$ is defined as in \eqref{Hm}  and 
\begin{equation}\label{Im}
I_m=\int_{\R^{2m}} \eta_0 \lm \eta_0\: dy=  - \frac{m 4^{2m}}{\beta^*\omega_{2m}}\int_{\R^{2m}} \frac{\log\bra{1+\frac{|y|^2}{4}}}{\bra{4+|y|^2}^{2m}}dy
\end{equation}
is as in \eqref{Imbrutto}.
\end{lemma}
\begin{proof}
First, using    a straightforward change of variable and the representation of $\mathbb{S}^{2m}$ through the standard stereographic projection, we observe that 
\[
\int_{\R^{2m}} e^{2\beta^*\eta_0} dy=  \int_{\R^{2m}} \frac{4^m}{(1+|y|^2)^{2m}}dy = \omega_{2m}.
\]
Since $e^{2\beta^*\eta_0}= O(\frac{1}{|y|^{4m}})$ as $|y|\to +\infty$, we get \eqref{propeta1}. 

The proof of \eqref{propeta2} relies on the integration by parts formula of Proposition \ref{parts}. For any $1\le l\le m-1$, we have
\[
\Delta^l \eta_0(y)= \frac{m}{\beta^*} \sum_{k=0}^l a_{k,l} \frac{|y|^{2k}}{(4+|y|^2)^{2l}}, \quad a_{k,l}=  (-1)^l (l-1)!  \binom{l}{k} \frac{(m+l-1)! (m-l+k-1)!}{(m+k-1)!(m-l-1)!} 2^{4l-2k},
\]
and 
\[
\Delta^{l+\frac{1}{2}} \eta_0(y)= \frac{m}{\beta^*} \sum_{k=0}^l b_{k,l} \frac{|y|^{2k} y }{(4+|y|^2)^{2l+1}}  , 
\quad b_{k,l}= \begin{Si}{cl }
 8(k+1)a_{k+1,l}+ (2k-4l) a_{k,l} & 0\le k \le l-1, \\
 -2l a_{ll} & k = l.
\end{Si}
\]
Note that $a_{ll}=-2\tilde K_{m,l}$, where $\tilde K_{m,l}$ is as in \eqref{Kml1}. In any case, for $1\le j\le 2m-1$, we find
\begin{equation}\label{asym}
\Delta^\frac{j}{2} \eta_0  =-\frac{2m}{\beta^*} K_{m,\frac{j}{2}} \frac{e_j(y)}{|y|^j}  + O(|y|^{-2-j}),
\end{equation} 
as $|y|\to +\infty$, where $K_{m,\frac{j}{2}}$ and $e_{j}$ are defined  as in \eqref{Kml2} and \eqref{ej}.
Integrating by parts, we find 
\[
\int_{B_R(0)}|\Delta^\frac{m}{2}\eta_0|^2 dy= \int_{B_R(0)} \eta_0 \lm \eta_0 \, dy - \sum_{j=0}^{m-1}\int_{\partial B_R(0)} (-1)^{j+m} \nu \cdot \Delta^{\frac{j}{2}} \eta_0 \Delta^{\frac{2m-j-1}{2}} \eta_0 \, d\sigma.
\]
On $\partial B_R(0)$, \eqref{asym} and the  identity $\frac{2m}{\beta^*}K_{m,\frac{2m-1}{2}}= \frac{(-1)^{m-1}}{\omega_{2m-1}}$ imply 
\[\begin{split}
\nu \cdot \eta_0 \Delta^{\frac{2m-1}{2}} \eta_0   & =\bra{-\frac{2m}{\beta^*} \log \frac{R}{2} +O(R^{-2})} \bra{\frac{-2m}{\beta^*}K_{m,\frac{2m-1}{2}}R^{1-2m}+O(R^{-2m-1})} \\
&= \frac{(-1)^m  }{\omega_{2m-1}} R^{1-2m}  \bra{-\frac{2m}{\beta^*}\log \frac{R}{2} +  O(R^{-2}\log R)},
\end{split}
\]
and, for $1\le j \le m-1$, that
\[\begin{split}
 \nu \cdot \Delta^{\frac{j}{2}} \eta_0 \Delta^{\frac{2m-j-1}{2}} \eta_0  &= \bra{-\frac{2m}{\beta^*}K_{m,\frac{j}{2}}R^{-j}+O(R^{-j-2})}  \bra{-\frac{2m }{\beta^*} K_{m,\frac{2m-j-1}{2}} R^{1+j-2m}+ O(R^{j-2m-1})} \\
& = \bra{\frac{2m}{\beta^*}}^2 K_{m,\frac{j}{2}} K_{m,\frac{2m-j-1}{2}} R^{1-2m} +O(R^{-2m-1}).
\end{split}
\]
Hence, we have
\begin{equation}\label{for}
\int_{B_R(0)}|\Delta^\frac{m}{2}\eta_0|^2 dy = \int_{B_R(0)} \eta_0 \lm \eta_0 \, dy +\frac{2m}{\beta^*}\log \frac{R}{2}  -H_m +O(R^{-2}\log R).
\end{equation}
Finally, since $\eta_0 \lm \eta_0$ decays like $|y|^{-4m}\log |y|$ as $|y|\to +\infty$, we get
\[
\int_{B_R(0)}\eta_0 \lm \eta_0 \,dy=  I_m +O(R^{-2m} \log R),
\]
which, together with \eqref{for}, gives the conclusion. 
\end{proof}

\begin{rem}\label{rem integrals}
Proposition \ref{conv eta} and Lemma \ref{int eta0} imply
\begin{enumerate}
\item $\dis{\lim_{n\to +\infty}\int_{B_{R r_n(x_n) }} \lambda_n u_n^2 e^{\beta_n u_n^2} dx =1+O(R^{-2m})}$.
\item $\dis{\lim_{n\to +\infty}\int_{B_{R r_n(x_n) }} \lambda_n  \mu_n u_n e^{\beta_n u_n^2} dx =1+O(R^{-2m})}$.
\item $\dis{\lim_{n\to + \infty}\int_{B_{R r_n(x_n) }} \lambda_n \mu_n |u_n| e^{\beta_n u_n^2} dx =1+O(R^{-2m})}$.
\item $\dis{\lim_{n\to +\infty}\int_{B_{R r_n(x_n) }} \lambda_n  \mu_n^2  e^{\beta_n u_n^2} dx =1+O(R^{-2m})}$.
\end{enumerate}
Indeed, all the integrals converge to $\dis{\omega_{2m}^{-1}\int_{B_R(0)} e^{2\beta^* \eta_0} dy}$. 
\end{rem}

\subsection{Estimates on the derivatives of \texorpdfstring{$\boldsymbol{u_n}$}{un} }
In this subsection, we prove some pointwise estimates on $u_n$ and its derivatives that  are inspired from the ones in Theorem 1 of \cite{mar} and Proposition 11 of \cite{MS} (where the authors assume $\alpha=0$ and $u_n\ge0$).

\begin{prop}\label{stima Luca}
There exists a constant $C>0$, such that  
$$
|x-x_n|^{2m}\lambda_n u_n^2 e^{\beta_n u_n^2} \le C,
$$
for any $x\in \Omega$. 
\end{prop}
\begin{proof}
Let us denote 
\begin{equation}\label{defln1}
L_n:= \sup_{x\in  \ov{\Omega}} |x-x_n|^{2m} \lambda_n u_n^2(x)e^{\beta_n u_n^2(x)}.
\end{equation}
Assume by contradiction that $L_n\to +\infty$ as $n\to +\infty$. Take a point $\tilde x_n\in \Omega$ such that 
\begin{equation}\label{defln2}
L_n =  |\tilde x_n -x_n |^{2m} \lambda_n u_n^2(\tilde x_n) e^{\beta_n  u_n^2(\tilde x_n)}, 
\end{equation}
and define $\tilde \mu_n := u_n(\tilde x_n)$ and $s_n\in \R^+$ such that 
\begin{equation}\label{defsn}
\omega_{2m} s_n^{2m}  \lambda_n \tilde \mu_n^2 e^{\beta_n \tilde \mu_n^2}= 1. 
\end{equation}
We will show that $\tilde x_n$ and $s_n$ satisfy the assumptions of Proposition \ref{conv eta gen}. Clearly, since $L_n\to+\infty$, \eqref{defln2} and \eqref{defsn} imply that  
\begin{equation}\label{ratio}
|\tilde \mu_n| \to +\infty\qquad \text{ and } \qquad \frac{|x_n-\tilde x_n|}{s_n}\to +\infty.
\end{equation} 
In particular, $s_n\to 0$. Let $v_n$ and $\tilde \Omega_n$ be as in Proposition \ref{conv eta gen}. Using \eqref{defln1} and \eqref{defln2}, we obtain 
\begin{equation}\label{important}
\frac{v_n^2}{\tilde \mu_n^2} e^{v_n^2- \tilde \mu_n^2} \le \frac{|y_n|^{2m}}{|y-y_n|^{2m}},
\end{equation}
for any $y\in \tilde \Omega_{n}$, where $y_n:= \frac{x_n -\tilde x_n}{s_n}$.  Since  $|y_n|\to+\infty$,  \eqref{important} yields
\begin{equation}\label{important2}
\frac{v_n^2}{\tilde \mu_n^2} e^{v_n^2- \tilde \mu_n^2} \le C(R) \qquad \text{in } \tilde \Omega_n\cap B_R(0),
\end{equation}
for sufficiently large $n$. Thanks to \eqref{important2}, we infer that
$$
\left|\frac{v_n}{\tilde \mu_n}\right| \le C(R) \quad \text{ and } \quad v_n^2 - \tilde \mu_n^2 \le C(R)
$$
on the set $\{|v_n|\ge |\tilde \mu_n|\}\cap B_R(0)$, and therefore on $\tilde \Omega_n\cap B_R(0)$. Then, all the assumptions of Proposition \ref{conv eta gen} are satisfied. In particular, as in Remark \ref{rem integrals}, by Proposition \ref{conv eta gen} and Lemma \ref{int eta0}, we get 
\begin{equation}\label{mass}
\lim_{n\to+\infty}\int_{B_{Rs_n}(\tilde x_n)} \lambda_n u_n^2 e^{\beta_n u_n^2}dx =  \omega_{2m}^{-1}\int_{B_R(0)}e^{2\beta^*\eta_0} dy = 1 +O(R^{-2m}). 
\end{equation}
Besides, if $r_n$ is as in \eqref{rn}, we have $r_n\le s_n$ and, by \eqref{ratio}, $B_{Rs_n}(\tilde x_n)\cap B_{Rr_n}(x_n)=\0$, for any $R>0$. Then, \eqref{lambdan}, Remark \ref{rem integrals}, and  \eqref{mass}  imply
$$
1=\lim_{n\to +\infty} \int_{\Omega} \lambda_n u_n^2 e^{\beta_n u_n^2} dx \ge \lim_{n\to +\infty} \int_{B_{R r_n}(x_n)\cup B_{Rs_n}(\tilde x_n)} \lambda_n u_n^2 e^{\beta_n u_n^2} dx   =   2 +O(R^{-2m}),
$$
which is a contradiction for large values of $R$. 
\end{proof}

Next, we prove pointwise estimates on $|\nabla^l u_n|$ for any $1\le l\le 2m-1$.

\begin{prop}\label{stime grad}
There exists a constant $C>0$ such that 
$$
|x-x_n|^l | u_n \nabla^l u_n | \le C,
$$
for any $x\in \Omega$ and $1\le l\le 2m-1$. 
\end{prop}

The proof of Proposition \ref{stime grad} follows the same steps of the ones of Propositions \ref{conv eta gen}. However, in this case it will be more difficult to obtain uniform bounds on $u_n$ on a small scale. For any $1\le l \le 2m-1$, we denote 
\begin{equation}\label{Lnl}
L_{n,l}:= \sup_{x\in \ov \Omega} |x-x_n|^l | u_n| |\nabla^l u_n |.
\end{equation}
Let $ x_{n,l} \in \Omega$ be such that
\begin{equation}\label{Lnl2}
| x_{n,l} -x_n|^l |u_n( x_{n,l}) \nabla^l u_n ( x_{n,l})|=L_{n,l}.
\end{equation}
We define $s_{n,l} := | x_{n,l} -x_n |$, $ \mu_{n,l}:=u_n( x_{n,l})$, and $y_{n,l} := \frac{x_n- x_{n,l}}{s_{n,l}}$. Up to subsequences,  we can assume  $y_{n,l} \to \ov{y}_l\in \mathbb S^{2m-1}$ as $n\to +\infty$.  Consider now the scaled functions
\[
v_{n,l} (y) =  u_{n} ( x_{n,l} + s_{n,l} y),
\]
which are defined on the sets $\Omega_{n,l}:=\{ y\in \R^{2m}\;:\: x_{n,l} + s_{n,l}y \in \Omega  \}$. Observe that $v_{n,l}$ satisfies 
\begin{equation}\label{eqvnl}
\begin{Si}{cc}
\lm v_{n,l} = s_{n,l}^{2m} \lambda_n v_{n,l} e^{\beta_n v_{n,l}^2}   +s_{n,l}^{2m} \alpha v_{n,l} & \text{ in }\Omega_{n,l},\\
v_{n,l}= \partial_{\nu} v_{n,l} = \ldots = \partial_{\nu}^{m-1} v_{n,l}=0, & \text{ on }\partial \Omega_{n,l}. 
\end{Si}
\end{equation}
Moreover, Proposition \ref{stima Luca} yields 
\begin{equation}\label{scaled estimate}
s_{n,l}^{2m } \lambda_n v_{n,l}^{2} e^{\beta_n v_{n,l}^2} \le \frac{C}{|y-y_{n,l}|^{2m}},
\end{equation}
for any $y\in \Omega_{n,l}$, and   \eqref{Lnl2} can be rewritten as
\begin{equation}\label{Lnl3}
L_{n,l}= |v_{n,l}(0) ||\nabla^l v_{n,l}(0)| =  |\mu_{n,l}| |\nabla^l v_{n,l}(0)|.
\end{equation}

\begin{rem}\label{rem sn}
If $L_{n,l}\to +\infty$ as $n\to +\infty$, then Lemma \ref{conv0C} implies that $s_{n,l} \to 0$. In particular,  \eqref{scaled estimate} gives
\[
s_{n,l}^{2m } \lambda_n v_{n,l} e^{\beta_n v_{n,l}^2} \to 0
\]
as $n\to +\infty$, uniformly in $\Omega_{n,l} \setminus B_{\frac{1}{R}}(\ov y_l)$, for any $R>0$. Indeed, if we choose a sequence $\{a_n\}_{n\in \N}$ such that $a_n \to +\infty$ and $s_{n,l}^{2m}\lambda_n a_n e^{\beta_n a_n^2} \to 0$ as $n\to +\infty$, then we have 
$$
\left|s_{n,l}^{2m } \lambda_n v_{n,l} e^{\beta_n v_{n,l}^2} \right|\le s_{n,l}^{2m } \lambda_n a_ne^{\beta_n a_n^2},
$$
on the set $\{|v_{n,l}|\le a_n\}$, while \eqref{scaled estimate} gives
$$
\left|s_{n,l}^{2m } \lambda_n v_{n,l} e^{\beta_n v_{n,l}^2} \right|\le  \frac{s_{n,l}^{2m } \lambda_n v_{n,l}^2 e^{\beta_n v_{n,l}^2} }{a_n} \le \frac{C}{a_n|y-y_{n,l}|},
$$
on the set $\{|v_{n,l}|\ge a_n\}$. 
\end{rem}

In the following, we will treat separately the cases $l=1$ and $2\le l \le 2m-1$. 

\begin{lemma}\label{grad dist}
If $L_{n,1}\to +\infty$ as $n\to +\infty$, then we have $\frac{d( x_{n,1},\partial \Omega)}{s_{n,1}} \to +\infty $.  Moreover,  $\frac{v_{n,1}}{\mu_{n,1}} \to 1$ in $C^{2m-1,\gamma}_{loc}(\R^{2m}\setminus \{\ov y_1\})$, for any $\gamma\in (0,1)$. 
\end{lemma}
\begin{proof}
It is sufficient to prove that $x_{n,1}$, $s_{n,1}$ and $v_{n,1}$ satisfy the assumptions of Lemma \ref{boundarydist0} and Lemma \ref{boundarydist}, with $\Sigma=\{\ov y_1\}$.  First of all, we observe that, for any $R>0$, the definition of $L_{n,1}$ implies $|\nabla v_{n,1}^2|\le C(R)L_{n,1}$ in $\Omega_{n,1} \setminus B_\frac{1}{R}(\ov y_1)$. Then, a Taylor expansion and \eqref{Lnl3} yield
\begin{equation}\label{L1}
v_{n,1}^2 \le \mu_{n,1}^2+ C(R) L_{n,1}\le C(R) D_{n,1}^2
\end{equation}
in $\Omega_{n,1} \cap B_{R}(0)\setminus B_\frac{1}{R}(\ov y_1)$, where $D_{n,1}:= \max_{0\le i\le 2m-1}|\nabla^i v_{n,1}(0)|$.  Moreover, by equation \eqref{eqvnl},  Remark \ref{rem sn}, and  \eqref{L1}, we get  
$$|\lm v_{n,1} | = o(1) + s_{n,1}^{2m} \alpha v_{n,1} =  o(1) + O(s_{n,1}^{2m}D_{n,1}),$$
uniformly in $\Omega_{n,1} \cap B_{R}(0)\setminus B_\frac{1}{R}(\ov y_1)$. Finally, Remark \ref{rem sn} gives $s_{n,1}\to 0$, while \eqref{Lnl3} and the condition $L_{n,1}\to +\infty$ imply $D_{n,1}\to +\infty$.
\end{proof}

We can now prove Proposition \ref{stime grad} for $l=1$. 

\begin{proof}[Proof of  Proposition \ref{stime grad} for $l=1$]   Assume by contradiction that $L_{n,1}\to +\infty$, as $n\to +\infty$.  Consider the function $z_n(y):= \dfrac{v_{n,1}(y)-  \mu_{n,1}}{| \nabla v_{n,1}(0)|}$. On the one hand, by the definitions of $L_{n,1}$ and $ x_{n,1}$ in \eqref{Lnl} and \eqref{Lnl2}, and by Lemma \ref{grad dist},  we have 
$$
|\nabla v_{n,1}(y)|\le  \frac{|\nabla v_{n,1}(0)|(1+o(1))}{|y-y_{n,1}|} \le C(R) |\nabla v_{n,1}(0)|,
$$
uniformly in $B_R(0)\setminus B_\frac{1}{R}(\ov y_{1})$, for any $R>0$. In particular,
\[
|\nabla z_n(y) |
\le C(R)   \qquad \text{ in } B_{R}(0)\setminus B_\frac{1}{R}(\ov y_1).
\]
Since $z_n(0)=0$,  $z_n$ is bounded in $L^\infty_{loc}(\R^{2m}\setminus \{\ov y_1\})$.  On the other hand, arguing as in \eqref{explanation}, Lemma \ref{grad dist}  implies  that
$$s_{n,1}^{2m} \mu_{n,1}^2 =o(1),$$
and, using also \eqref{scaled estimate}, that
$$
\lm z_n = \frac{\lambda_n s_{n,1}^{2m} v_{n,1} e^{\beta_n v_{n,1}^2} + \alpha s_{n,1}^{2m} v_{n,1}}{|\nabla v_{n,1} (0)|} = O\bra{\frac{1}{ \mu_{n.1} |\nabla v_{n,1}(0)|}}=o(1),  \qquad \text{ in } B_{R}(0)\setminus B_\frac{1}{R}(\ov y_1).
$$
By Proposition \ref{ell loc}, we find a function $z_0$, harmonic in $\R^{2m}\setminus \{ \ov y_1\}$, such that, up to subsequences, $z_n\to z_0$ in $C^{2m-1,\gamma}_{loc}(\R^{2m} \setminus \{\ov  y_{1}\})$, for any $\gamma\in (0,1)$. We claim now that $z_0$ must be constant on $\R^{2m} \setminus \{ \ov y_{1}\}$. To prove this, we observe that, by Lemma \ref{integral est}, for any $R>0$ there exists a constant $C(R)>0$ such that 
$$
\| \nabla v_{n,1}^2\|_{L^1(B_R(0))} \le C(R).
$$
Applying Lemma \ref{grad  dist} and \eqref{Lnl3}, we obtain
\[\begin{split}
\| \nabla v_{n,1}^2\|_{L^1(B_R(0))} &\ge 2 \int_{B_{R}(0)\setminus B_\frac{1}{R}(\ov y_1)} |v_{n,1}| |\nabla v_{n,1}| dy \\ 
&= 2|{ \mu_{n,1}}| (1+o(1)) \|\nabla v_{n,1}\|_{L^1(B_{R}(0)\setminus B_\frac{1}{R}(\ov y_1))} \\ 
& =2 L_{n,1} (1+o(1)) \|\nabla z_n\|_{L^1(B_{R}(0)\setminus B_\frac{1}{R}(\ov y_1))}.
\end{split}
\]
Thus, as $n\to +\infty$, we have
$$
\|\nabla z_n\|_{L^1(B_{R}(0)\setminus B_\frac{1}{R}(\ov y_1))}  \le \frac{C(R)}{L_{n,1}}\to 0.
$$
Hence, $z_0$ must be constant, which contradicts
$$
|\nabla z_0(0)| = \lim_{n\to +\infty} |\nabla z_n(0)| =1.
$$

\end{proof}

We shall now deal with the case $2\le l\le 2m-1$. Since Proposition \ref{stime grad} has been proved for $l=1$, we know that $L_{n,1}$ is bounded, i.e.  
$$
|x-x_n| |u_n(x)| |\nabla u_n (x)| \le C,
$$ 
for any $x\in \Omega$. Equivalently, given any $1\le l\le 2m-1$, we have 
\begin{equation}\label{new}
|v_{n,l}(y)||\nabla v_{n,l}(y)| \le \frac{C}{|y-y_{n,l}|},
\end{equation}
for any $y\in \Omega_{n,l}$. In particular, \eqref{new} yields
\begin{equation}\label{new2}
\|\nabla v_{n,l}^2\|_{L^\infty(\Omega_{n,l}\setminus B_{\frac{1}{R}}(\ov y_l))}\le C(R),
\end{equation}
for any $R>0$.

\begin{lemma}\label{grad dist case l}
Fix any  $2\le l\le 2m-1$. If $L_{n,l}\to +\infty$ as $n\to +\infty$, then we have $\frac{d( x_{n,l},\partial \Omega)}{s_{n,l}} \to +\infty $.  Moreover,  $\frac{v_{n,l}}{\mu_{n,l}} \to 1$ in $C^{2m-1,\gamma}_{loc}(\R^{2m}\setminus \{\ov y_l\})$, for any $\gamma\in (0,1)$. 
\end{lemma}
\begin{proof}
As in Lemma \ref{grad dist}, we show that $x_{n,l}$, $s_{n,l}$ and $v_{n,l}$ satisfy the assumptions of Lemma \ref{boundarydist0} and Lemma \ref{boundarydist}, with $\Sigma=\{\ov y_l\}$.  Let us denote $D_{n,l}:= \max_{0\le i\le 2m-1}|\nabla^i v_{n,l}(0)|$.  Note that \eqref{Lnl3} and the condition $L_{n,l}\to +\infty$ imply $D_{n,l}\to +\infty$.  Then, for any $R>0$, a Taylor expansion and \eqref{new2} yield
\begin{equation}\label{L1 case l}
v_{n,l}^2 \le \mu_{n,l}^2+ C(R) \le C(R) D_{n,l}^2
\end{equation}
in $\Omega_{n,l} \cap B_{R}(0)\setminus B_\frac{1}{R}(\ov y_l)$.  Moreover, by equation \eqref{eqvnl},  Remark \ref{rem sn}, and  \eqref{L1 case l}, we get  
$$|\lm v_{n,l} | = o(1) + s_{n,l}^{2m} \alpha v_{n,l} =  o(1) + O(s_{n,l}^{2m}D_{n,l}),$$
uniformly in $\Omega_{n,l} \cap B_{R}(0)\setminus B_\frac{1}{R}(\ov y_l)$. 
 \end{proof}

\begin{proof}[Proof of Proposition \ref{stime grad} for $2\le l\le 2m -1$.] Assume by contradiction that $L_{n,l}\to +\infty$ as $n\to +\infty$. Consider the function $z_n:= \frac{v_{n,l}- \mu_{n,l}}{|\nabla^l v_n(0)|}$. Observe that \eqref{Lnl3}, \eqref{new}, and Lemma \ref{grad dist case l}, yield
\begin{equation}\label{grad fin}
|\nabla z_n (y)|\le \frac{C(R)}{L_{n,l}} \to 0,
\end{equation}
uniformly in $B_R(0)\setminus B_\frac{1}{R}(\ov y_l)$, for any $R>0$. Since $z_n(0)=0$, \eqref{grad fin} implies that 
\[
|z_n|\le \frac{C(R)}{L_{n,l}}\to 0,
\]  uniformly in $B_R(0)\setminus B_\frac{1}{R}(\ov y_l)$. Similarly, as a consequence of equation \eqref{eqvnl}, \eqref{scaled estimate}, and Lemma  \ref{grad dist case l}, one has 
$$
|\lm z_n |\le \frac{C(R)}{L_{n,l}},
$$
in $B_R(0)\setminus B_\frac{1}{R}(\ov y_l)$. Therefore, up to subsequnces,  $z_n\to 0$  in $C^{2m-1,\gamma}_{loc}(\R^{2m}\setminus \{\ov y_{l}\})$,  for any $\gamma\in (0,1)$. Since $|\nabla^l z_n(0)| = 1$ for any $n$, we get a contradiction.  
\end{proof}

\subsection{Polyharmonic truncations}

In this subsection, we will generalize the truncation argument introduced in \cite{AD} and \cite{Li1}. For any $A>1$ and $n\in \N$, we will introduce a new function $u_n^A$ whose values are close to $\frac{\mu_n}{A}$ in a small ball centered at $x_n$, and which coincides with $u_n$ outside the same ball.   

\medskip
\begin{lemma}\label{def rhon}
For any $A>1$ and $n\in \N$, there exists a radius $0<\rho_n^A <d(x_n,\partial \Omega)$ and a constant $C=C(A)$ such that
\begin{enumerate}
\item $u_n\ge  \frac{\mu_n}{A}$ in $B_{\rho^A_n}(x_n)$.
\item $|u_n - \frac{\mu_n}{A}|\le C \mu_n^{-1}$  on $\partial B_{\rho^A_n}(x_n)$.
\item $|\nabla^l u_n |\le \dfrac{C}{\mu_n (\rho_n^A)^l}$  on $\partial B_{\rho_n^A}(x_n)$, for any $1\le l\le 2m-1$.
\item If $r_n$ is defined as in \eqref{rn}, then $\frac{\rho_n^A}{r_n} \to +\infty$ as $n\to +\infty$. 
\end{enumerate}
\end{lemma}
\begin{proof}
For any $\sigma\in \mathbb S^{2m-1}$, the function $t\mapsto u_n(x_n + t \sigma)$ ranges from $\mu_n$ to $0$ in the interval $[0,t^*_n(\sigma)]$, where $t^*_n(\sigma) := \sup\{ t>0 \;:\; x_n +s \sigma \in \Omega \text{ for any } s\in [0,t] \}$. Since $u_n\in C(\ov\Omega)$, one can define 
$$
t_n^A(\sigma):= \inf\{t \in [0,t^*_n(\sigma))\;:\; u_n(x_n +t\sigma)=\frac{\mu_n}{A}\}. 
$$ 
Clearly, one has $0<t_n^A(\sigma)<t_n^*(\sigma)$ and $u_n(x_n +t_n^A(\sigma) \sigma) = \frac{\mu_n}{A}$, for any $\sigma \in \mathbb S^{2m-1}$. Moreover,  the function $\sigma\longmapsto t_n^A(\sigma)$ is lower semi-continuous on $\mathbb S^{2m-1}$. In particular, we can find $\ov{\sigma}_n^A$ such that $\dis{t_n^A(\ov \sigma_n^A)=\min_{\sigma\in \mathbb S^{2m-1}} t_n^A(\sigma)}$. We  define $\rho_n^A:=t_n^A(\ov \sigma_n)$, and $y_n^A:= x_n + \rho_n^A\ov \sigma_n^A \in \partial B_{\rho_n^A}(x_n)$. By construction we have, $0<\rho_{n}^A<d(x_n,\partial \Omega)$,  $u_n\ge \frac{\mu_n}{A}$ on $B_{\rho_n^A}(x_n)$, and $u_n(y_n^A)=\frac{\mu_n}{A}$. Thus, applying Proposition \ref{stime grad}, we get 
$$
|\nabla^l u_n|\le \frac{C A}{\mu_n (\rho_n^A)^l},
$$ 
on $\partial B_{\rho_n^A} (x_n)$, for any $1\le l\le 2m-1$. Furthermore, for any $x\in \partial B_{\rho_n^A} (x_n)$, one has
$$
|u_n(x) - \frac{\mu_n}{A}| = |u_n(x) - u_n(y_n^A)| \le \pi \rho_n^A \sup_{\partial B_{\rho_n^A} (x_n)} |\nabla u_n|  \le \frac{C}{\mu_n}. 
$$
Finally, if $r_n$ is as in \eqref{rn}, Proposition \ref{conv eta} and \eqref{etan} imply that $u_n = \mu_n + O(\mu_n^{-1})$ uniformly in $B_{r_nR}(x_n)$, for any $R>0$. Therefore, for sufficiently large $n$, we have $r_n R < \rho_n^A$. Since $R$ is arbitrary, we get the conclusion. 
\end{proof}

Let $\rho_n^A$ be as in the previous lemma and let  $v_n^A \in C^{2m} (\ov{B_{\rho_n^A}(x_n)})$ be the unique solution of 
$$
\begin{Si}{ll}
\lm v_n^A =0  &  \text{ in }  B_{\rho_n^A}(x_n), \\
  \partial^i_\nu v_n^A  =  \partial^i_\nu u_n  &    \text{ on } \partial B_{\rho_n^A}(x_n), 0\le i \le m-1.
\end{Si}
$$
We consider the function 
\begin{equation}\label{unA}
u_n^A(x) := \begin{Si}{cl}
v_n^A &   \text{ in }  B_{\rho_n^A}(x_n),\\
u_n &    \text{ in } \Omega \setminus B_{\rho_n^A}(x_n).
\end{Si}
\end{equation}
By definition, we have $u_n^A\in H^{m}_0(\Omega)$. The main purpose of this section is to study the properties of $u_n^A$.

\begin{lemma}\label{est unA}
For any $A>1$, we have 
$$
u_n^A = \frac{\mu_n}{A} + O(\mu_n^{-1}),
$$
uniformly on $\ov{B_{\rho_n^A}(x_n)}$. 
\end{lemma}
\begin{proof}
Define $\tilde v_n(y) := v_n^A ( x_n + \rho_n^A y) -\frac{\mu_n}{A}$ for $y\in B_1(0)$. Then, by elliptic estimates (Proposition \ref{ell harm2}), we have
\[\begin{split}
\|v_n^A-\frac{\mu_n}{A}\|_{L^\infty(B_{\rho_n^A}(x_n))} = \|\tilde v_n\|_{L^\infty(B_{1}(0))} 
& \le C\sum_{l=0}^{m-1} \|\nabla^l \tilde v_n\|_{L^\infty(\partial B_1(0))} \\
& = C \sum_{l=0}^{m-1} (\rho_n^A)^l\|\nabla^l  v_n^A\|_{L^\infty(\partial B_{\rho_n^A}(x_n))}\\
& = C \sum_{l=0}^{m-1} (\rho_n^A)^l\|\nabla^l  u_n\|_{L^\infty(\partial B_{\rho_n^A}(x_n))}
\end{split}
\]
By Lemma \ref{def rhon}, we know  that $(\rho_n^A)^l\|\nabla^l  u_n\|_{L^\infty(\partial B_{\rho_n^A}(x_n))}\le \frac{C}{\mu_n}$ and the proof is complete.
\end{proof}

\begin{prop}\label{trunc}
For any $A>1$, we have
$$
\limsup_{n\to +\infty} \int_{\Omega} |\Delta^\frac{m}{2} u_n^A|^2 dx \le \frac{1}{A}. 
$$
\end{prop}
\begin{proof}
Since $u_n^A \equiv u_n$ in $\Omega \setminus B_{\rho_n^A}(x_n)$, $u_n^A$ is $m-$harmonic in $B_{\rho_n^A}(x_n)$,  and $\partial_\nu^j u_n^A =\partial_\nu^j u_n$ on $\partial B_{\rho_n^A}(x_n)$ for $0\le j\le m-1$, we have 
\begin{equation}
\begin{split}\label{A1}
\int_{\Omega} |\Delta^\frac{m}{2} (u_n  -u_n^A)|^2 dx &  = \int_{B_{\rho_n^A}(x_n)} \Delta^\frac{m}{2} (u_n - u^A_n) \Delta^\frac{m}{2} u_n \,dx\\
& = \int_{B_{\rho_n^A}(x_n)}  (u_n - u_n^A) \lm u_n \,dx. 
\end{split}
\end{equation}
As a consequence of Lemma \ref{def rhon}, we get $\lm u_n\ge 0$ in $B_{\rho_n^A}(x_n)$. Therefore, the maximum principle guarantees $u_n\ge u_n^A$ in $B_{\rho_n^A}(x_n)$. Hence, if $r_n$ is as in \eqref{rn}, we have  
\begin{equation}\label{A2}
\begin{split}
\int_{B_{\rho_n^A}(x_n)}  (u_n - u^A_n) \lm u_n dx &\ge \int_{B_{Rr_n}(x_n)}  (u_n - u^A_n) \lm u_n dx \\
&\ge \int_{B_{Rr_n}(x_n)}  (u_n - u^A_n) \lambda_n u_n e^{\beta_n u_n^2} dx,
\end{split}
\end{equation}
for any $R>0$. By Lemma \ref{est unA}, \eqref{rn}, and Proposition \ref{conv eta}, we find 
\begin{equation}\label{A3}
\begin{split}
\int_{B_{Rr_n}(x_n)}  &(u_n - u^A_n) \lambda_n u_n e^{\beta_n u_n^2} dx \\& = r_n^{2m}\lambda_n \int_{B_{R}(0)}  \bra{\mu_n +\frac{\eta_n}{\mu_n}-\frac{\mu_n}{A} + O(\mu_n^{-1}) } \bra{\mu_n +\frac{\eta_n}{\mu_n}} e^{\beta_n \bra{ \mu_n^2 +2 \eta_n + \frac{\eta_n^2}{\mu_n^2}}} dy\\
&= \omega_{2m}^{-1}\bra{ 1- \frac{1}{A}} \int_{B_{R}(0)} e^{2\beta^*\eta_0} dy + o(1),
\end{split}
\end{equation}
where $\eta_n$ and $\eta_0$ are as in \eqref{etan} and \eqref{eta0}. Using \eqref{A1}, \eqref{A2}, \eqref{A3}, and Lemma \ref{int eta0}, as $n\to+\infty$  and $R\to +\infty$ we find 
\begin{equation}\label{A4}
\liminf_{n\to +\infty} \int_{\Omega} |\Delta^\frac{m}{2} (u_n  -u_n^A)|^2 dx \ge 1-\frac{1}{A}. 
\end{equation} 
Finally, since $u_n^A$ is $m-$harmonic in $B_{\rho_n^A}(x_n)$, we have
\begin{equation}\label{A5}
\begin{split}
1+o(1)&=\int_{\Omega} |\Delta^\frac{m}{2}u_n|^2 dx \\& = \int_{\Omega} |\Delta^\frac{m}{2} u_n^A|^2dx +\int_{\Omega} |\Delta^\frac{m}{2}\bra{u_n -u_n^A}|^2dx  + 2\int_{\Omega} \Delta^\frac{m}{2}u_n^A \cdot \Delta^\frac{m}{2}(u_n-u_n^A)dx \\ 
& = \int_{\Omega} |\Delta^\frac{m}{2} u_n^A|^2dx +\int_{\Omega} |\Delta^\frac{m}{2}\bra{u_n -u_n^A}|^2dx.
\end{split}
\end{equation}
Thus, \eqref{A4} and \eqref{A5} yield the conclusion. 
\end{proof}

As a consequence of Proposition \ref{trunc}, we get some simple but crucial  estimates. 

\begin{lemma}\label{lemma crucial}  Let $0\le \alpha<\lambda_1(\Omega) $ and let $S_{\alpha,\beta^*}$ be as in \eqref{sup}. Then, we have
$$S_{\alpha,\beta^*} = |\Omega| +  \lim_{n\to +\infty} \frac{1}{\lambda_n \mu_n^2 }.$$
In particular, $\lambda_n \mu_n\to 0$ as $n\to +\infty$.  
\end{lemma}
\begin{proof}
Fix $A>1$ and let $u_n^A$ be as in \eqref{unA}. By Adams' inequality \eqref{Adams} and Proposition \ref{trunc}, we know that $e^{\beta_n (u_n^A)^2}$ is bounded in $L^{p}(\Omega)$, for any  $1<p<A$. Since $u_n^A\to 0$ a.e. in $\Omega$, Theorem \ref{Vitali} gives
\begin{equation}\label{outside}
\lim_{n\to +\infty}\int_{\Omega\setminus B_{\rho_n^A}(x_n)} e^{\beta_n u_n^2} dx  = \lim_{n\to +\infty} \int_{\Omega\setminus B_{\rho_n^A}(x_n)} e^{\beta_n (u_n^A)^2} dx  = |\Omega|.
\end{equation}
By Lemma \ref{def rhon}, $u_n \ge \frac{\mu_{n}}{A}$ on   $B_{\rho_n^A}(x_n)$. Hence,
\begin{equation}\label{inside1}
\int_{B_{\rho_n^A(x_n)}}  e^{\beta_n u_n^2} dx \le \frac{A^2}{\mu_n^2} \int_{B_{\rho_n^A}(x_n)} u_n^2 e^{\beta_n u_n^2}dx \le  \frac{A^2}{\lambda_n \mu_n^2}.
\end{equation}
Moreover, for $R>0$ large enough, Lemma \ref{def rhon}  and Remark \ref{rem integrals} imply 
\begin{equation}\label{inside2}
\limsup_{n\to+\infty} \int_{B_{\rho_n^A(x_n)}}  e^{\beta_n u_n^2} dx \ge  \limsup_{n\to+ \infty} \int_{B_{r_nR(x_n)}}  e^{\beta_n u_n^2} dx  = (1+O(R^{-2m})) \limsup_{n\to +\infty} \frac{1}{\lambda_n \mu_n^2}.
\end{equation}
From \eqref{extremal}, \eqref{outside}, \eqref{inside1},  \eqref{inside2}, and Lemma \ref{limsub}, we get 
$$
|\Omega| + \limsup_{n\to +\infty} \frac{1}{\lambda_n\mu_n^2} \le S_{\alpha,\beta^*} \le |\Omega| + A^2\liminf_{n\to +\infty} \frac{1}{\lambda_n\mu_n^2}.
$$
Since $A$ is an arbitrary number greater than $1$, we get the conclusion. 
\end{proof}

We conclude this section with the following lemma, which gives $L^1$ bounds on $\lm (\mu_n u_n)$. This  will be important in the analysis of the behaviour of $u_n$ far from $x_0$, which is given in the next section.    
\begin{lemma}\label{lemma mun} The sequence
$\lambda_n \mu_n u_n e^{\beta_n u_n^2}$ is bounded in $L^1(\Omega)$. Moreover,  $\lambda_n \mu_n  u_n e^{\beta_n u_n^2} \rw \delta_0 $ in the sense of measures. 
\end{lemma}
\begin{proof}
By Remark \ref{rem integrals},  it is sufficient to show that 
$$
\lim_{R\to 0}\limsup_{n\to +\infty} \lambda_n \int_{\Omega \setminus B_{r_n R}(x_n)} \mu_n |u_n| e^{\beta_n u_n^2} dx =0.
$$
Let us denote $f_n = \lambda_n \mu_n u_n e^{\beta_n u_n^2}$. Fix  $A>1$ and let $\rho_n^A$ and $u_n^A$ be as in Lemma \ref{def rhon} and \eqref{unA}. Then, for any $R>0$ and $n$ sufficiently large, we have 
$$
 \int_{\Omega \setminus B_{r_n R}(x_n)} | f_n(x)| dx =  \int_{B_{\rho_n^A}(x_n) \setminus B_{r_n R}(x_n)}|f_n(x)| dx + \int_{\Omega \setminus B_{\rho_n^A}(x_n)}  |f_n(x)| dx =: I_n^1+I_n^2. 
$$
By Lemma \ref{def rhon}, \eqref{lambdan}, and Remark \ref{rem integrals}, we obtain
\[\begin{split}
I_n^1 \le A \int_{B_{\rho_n^A}(x_n) \setminus B_{r_n R}(x_n)} \lambda_n u_n^2 e^{\beta_n u_n^2}  dx &\le A \int_{\Omega \setminus B_{r_n R}(x_n)} \lambda_n u_n^2 e^{\beta_n u_n^2} dx  \\ 
& = A \bra{1 - \int_{ B_{r_n R}(x_n)} \lambda_n u_n^2 e^{\beta_n u_n^2} dx} \\ 
&=  A \,O(R^{-2m}). 
\end{split}
\]
Therefore,
\begin{equation}\label{In1}
\dis{\limsup_{n\to +\infty} I_n^1 \le A\, O(R^{-2m})}.
\end{equation}  
For the second integral, we observe that Proposition \ref{trunc} and Adams' inequality imply that $e^{\beta_n (u_n^A)^2}$ is bounded in $L^p(\Omega)$, for any  $1<p<A$. In particular, applying H\"older's inequality and Lemma \ref{lemma crucial}, we get
\begin{equation}\label{In2}
\begin{split}
I_n^2 \le \int_{\Omega \setminus B_{\rho_n^A}(x_n)}  | f_n(x)| dx  &\le \lambda_n \mu_n \|e^{\beta_n (u_n^A)^2}\|_{L^p(\Omega)}\|u_n\|_{L^{\frac{p}{p-1}}(\Omega)} \\ & \le C \lambda_n \mu_n \|u_n\|_{L^{\frac{p}{p-1}}(\Omega)} \to 0,
\end{split}
\end{equation}
as $n\to +\infty$. Since $R$ is arbitrary, the conclusion follows from \eqref{In1} and \eqref{In2}. 
\end{proof}

\subsection{Convergence to  Green's fuction}
In this subsection, we will study the behavior of the sequence $\mu_n u_n$ according to the position of the blow-up point $x_0$.  First, we will  show that, if $x_0 \in \Omega,$ we have $\mu_n u_n \to G_{\alpha,x_0}$ locally uniformly in $\ov{\Omega} \setminus \{x_0\}$, where $G_{\alpha,x_0}$ is the Green's function for $\lm  -\alpha$, defined as in \eqref{green}.

\begin{lemma}\label{bound green}
The sequence $\mu_n u_n$ is bounded  in $W^{m,p}_0(\Omega)$,  for any $p \in [1,2)$.  
\end{lemma}
\begin{proof}
Let $v_n$ be the unique solution to 
\[
\begin{Si}{cc}
\lm  v_n = \lambda_n \mu_n u_n e^{\beta_n u_n^2}=:f_n & \mbox{ in }\Omega,\\
v_n = \partial_\nu v_n = \ldots = \partial^{m-1}_{\nu} v_n =0  & \mbox{ on }\partial \Omega.
\end{Si}
\]
By Lemma  \ref{lemma mun}, we know that $f_n$ is bounded in $L^1(\Omega)$.  
By Proposition \ref{ell L1}, we can conclude that $v_n$ is bounded in $W^{m,p}_0(\Omega)$ for any $1\le p<2$. Define now $w_n = \mu_n u_n -v_n$. Then $w_n$ solves
\[
\begin{Si}{cc}
\lm w_n = \alpha w_n + \alpha v_n &  \mbox{ in }\Omega,\\
w_n = \partial_\nu w_n =\ldots= \partial^{m-1} w_n =0 & \mbox{ on }\partial\Omega.
\end{Si} 
\]
If we test the equation against $w_n$,  using Poincare's and Sobolev's inequalities, we find that 
\[\begin{split}
\|w_n\|_{H^m_0(\Omega)}^2 = \alpha \|w_n\|_{L^2(\Omega)}^2 + \alpha\int_{\Omega} w_n v_n dx & \le \alpha \|w_n\|_{L^2(\Omega)}^2 + \alpha\|w_n\|_{L^2(\Omega)}\|v_n\|_{L^2(\Omega)}\\
&  \le \frac{\alpha}{\lambda_1(\Omega)}\|w_n\|_{H^m_0(\Omega)}^2 +\frac{\alpha}{\sqrt{\lambda_1(\Omega)}} \|w_n\|_{H^m_0(\Omega)} \|v_n\|_{L^2(\Omega)} \\
 & \le \frac{\alpha}{\lambda_1(\Omega)}\|w_n\|_{H^m_0(\Omega)}^2 +C \|w_n\|_{H^m_0(\Omega)}.
\end{split}
\]
Then,
$$
\|w_n\|_{H^m_0(\Omega)} \bra{1-\frac{\alpha}{\lambda_1(\Omega)}}\le C, 
$$
which implies that $w_n$ is bounded $H^{m}_0(\Omega)$. This yields the conclusion. 
\end{proof}

\begin{lemma}\label{conv Green}
Let $x_0$ be as in \eqref{mun and xn2}. If $x_0 \in \Omega$, then we have:
\begin{enumerate}
\item $\mu_n u_n \rw G_{\alpha,x_0}$  in $W^{m,p}_0(\Omega)$ for any $1< p<2$;
\item $ \mu_n u_n \to G_{\alpha,x_0}$ in $\dis{C^{2m-1,\gamma}_{loc}(\ov{\Omega}\setminus\{x_0\})}$. 
\end{enumerate}
\end{lemma}
\begin{proof}
Fix $1<p<2$. By Lemma \ref{bound green}, we can find $\tilde u\in W^{m,p}_0(\Omega)$ such that  $\mu_n u_n \rw \tilde{u}$ in $W^{m,p}_0(\Omega)$. Let $\ph$ be any test function in  $ C^\infty_c(\Omega)$. Applying Lemma \ref{lemma mun} and the compactness of the embedding of $W^{m,p}_0(\Omega)$ into $L^1(\Omega)$, we obtain 
\[\begin{split}
\int_{\Omega} (\mu_n\lambda_n u_n e^{\beta_n u_n^2} +\alpha \mu_n u_n ) \ph dx &=\ph(x_0) +\alpha \int_{\Omega} \tilde{u} \ph \,  dx + o(1).
\end{split}
\]
Hence necessarily $\tilde{u}=G_{\alpha,x_0}$. To conclude the proof, it remains to show  that $\mu_n u_n\to G_{\alpha,x_0}$ in $C^{2m-1,\gamma}_{loc}(\ov \Omega \setminus \{x_0\})$.  By elliptic estimates (Proposition \ref{ell new}), it is sufficient to show that $\lm (\mu_n u_n )$ is bounded in $L^s(\Omega\setminus B_\delta(x_0))$, for any $s>1$, $\delta>0$. This follows from Lemma \ref{conv0C} and Lemma \ref{lemma crucial}. 
\end{proof}

Lemma \ref{conv Green} describes the behaviour of $\mu_n u_n$ when $x_0\in \Omega$. The following Lemma deals with the case $x_0\in \partial \Omega$. In fact, we will prove in the next subsection that blow-up at the boundary is not possible. 

\begin{lemma}\label{conv 0 boun}
If $x_0 \in \partial \Omega$, we have:
\begin{enumerate}
\item $\mu_n u_n \rw 0$  in $W^{m,p}_0(\Omega)$ for any $1<p<2$. 
\item  $ \mu_n u_n \to 0$ in $\dis{C^{2m-1,\gamma}_{loc}(\ov{\Omega}\setminus\{x_0\})}$, for any $\gamma \in (0,1)$. 
\end{enumerate}
\end{lemma}
\begin{proof}
As before, using Lemma \ref{bound green} and Lemma \ref{lemma mun}, we can find $\tilde u \in W^{m,p}_0(\Omega)$, $p\in (1,2)$, such that $\mu_nu_n\rw \tilde u$ in $W^{m,p}_0(\Omega)$ for any $p\in (1,2)$ and $\mu_n u_n \to \tilde u $ in $C^{2m-1,\gamma}_{loc}(\ov \Omega\setminus \{x_0\})$, for any $\gamma \in (0,1)$. Moreover, as $n\to +\infty$, we have
\[\begin{split}
\int_{\Omega} (\mu_n\lambda_n u_n e^{\beta_n u_n^2} +\alpha \mu_n u_n ) \ph dx &= \alpha \int_{\Omega} \tilde{u} \ph \,  dx + o(1),
\end{split}
\]
Then, $\tilde u$ is a weak solution of $\lm \tilde u = \alpha \tilde u$ in $\Omega$. Since $\tilde u\in W^{m,p}_0(\Omega)$, elliptic regularity (Proposition \ref{ell zero}) implies $\tilde u\in {W^{3m,p}(\Omega)}$, for any $p\in (1,2)$. In particular, we have $\tilde u\in H^m_0(\Omega)$, and 
$$
\|\tilde u\|_{H^m_0(\Omega)}^2 = \alpha \|\tilde  u\|_{L^2(\Omega)}^2. 
$$
Since $0\le \alpha <\lambda_1(\Omega)$, we must have $\tilde u \equiv 0$. 
\end{proof}

\subsection{The Pohozaev identity and blow-up at the boundary}
In this subsection, we prove that the blow-up point $x_0$ cannot lie on $\partial \Omega$. The proof is based on the following Pohozaev-type identity.

\begin{lemma}\label{Pohozaev}
Let $\Omega\subseteq \R^{2m}$ be a bounded open set with Lipschitz boundary. If $u\in C^{2m}(\ov{\Omega})$ is a solution of 
\begin{equation}\label{eq Poho}
\lm u = h(u),
\end{equation}
with $h:\R \ra \R$ continuous, then for any $y\in \R^{2m}$ the following identity holds:
$$
\frac{1}{2}\int_{\partial\Omega} |\Delta^\frac{m}{2} u|^2 (x-y)\cdot \nu \,d\sigma(x) +\int_{\partial\Omega} f(x)d\sigma(x) = \int_{\partial \Omega} H(u(x)) (x-y) \cdot \nu d\sigma(x) - 2m \int_{\Omega} H(u(x))dx, 
$$
where $H(t):= \int_0^t h(s)ds$ and 
$$
f(x):= \sum_{j=0}^{m-1} (-1)^{m+j} \nu\cdot \bra{\Delta^{\frac{j}{2}} ((x-y)\cdot \nabla u) \Delta^{\frac{2m-j-1}{2}}u}.
$$
\end{lemma}
\begin{proof}
We multiply  equation \eqref{eq Poho} for $(x-y)\cdot\nabla u$ and integrate on $\Omega$ to obtain
\begin{equation}\label{first step Poho}
\int_{\Omega} (x-y)\cdot\nabla u \, \lm u \,dx = \int_{\Omega} (x-y)\cdot\nabla u \, h(u)dx. 
\end{equation}
On the one hand, using the divergence Theorem, we can rewrite the RHS of \eqref{first step Poho} as 
\[
\begin{split}
\int_{\Omega} (x-y)\cdot\nabla u \, h(u)dx & = \int_{\Omega} (x-y)\cdot\nabla H(u)dx  \\ 
& = \int_{\Omega} \dv \bra{ (x-y) H(u)}dx - 2m \int_{\Omega} H(u) dx  \\
& = \int_{\partial\Omega} H(u) (x-y)\cdot \nu \, d\sigma (x) - 2m \int_{\Omega} H(u) dx  . 
\end{split}
\]
On the other hand, we can integrate by parts the LHS  of \eqref{first step Poho} to find 
$$
\int_{\Omega} (x-y)\cdot\nabla u \, \lm u \, dx = \int_{\Omega} \Delta^\frac{m}{2} \bra{ (x-y) \cdot \nabla u } \Delta^\frac{m}{2} u \, dx + \int_{\partial\Omega} f d\sigma.  
$$
As proved in  Lemma 14 of \cite{MarPet}, we have the identity  
$$
 \Delta^\frac{m}{2} \bra{ (x-y) \cdot \nabla u } \cdot \Delta^\frac{m}{2} u  = \frac{1}{2} \dv \bra{ (x-y)|\Delta^\frac{m}{2} u|^2}.
$$
Hence, the divergence theorem yields 
$$
\int_{\Omega} (x-y)\cdot\nabla u \, \lm u \, dx = \frac{1}{2}\int_{\partial\Omega}  (x-y)\cdot \nu \, |\Delta^\frac{m}{2} u|^2 d\sigma(x) + \int_{\partial\Omega} f d\sigma.  
$$
\end{proof}

We now apply Lemma \ref{Pohozaev} to $u_n$ in a neighborhood of $x_0$, and we use Lemma \ref{conv 0 boun} to prove that $x_0$ must be in $\Omega$.  A smart choice of the point $y$ is crucial to control the boundary terms in the identity. This strategy was first introduced in \cite{RobWei} and was applied in \cite{MarPet} to Liouville equations in dimension $2m$. 

\begin{lemma}\label{no boundary}
Let $x_0$ be as in \eqref{mun and xn2}. Then $x_0 \in \Omega$.
\end{lemma}
\begin{proof}
We assume by contradiction that $x_0\in \partial \Omega$. If we fix a  sufficiently small $\delta >0$, we have that $\frac{1}{2}\le \nu \cdot \nu(x_0)\le 1$ on $\partial  \Omega \cap B_\delta(x_0)$. Then we can define  
\begin{equation*}
\rho_n := \frac{ \int_{\partial \Omega \cap B_\delta (x_0)}  |\Delta^\frac{m}{2} u_n|^2  (x-x_0)\cdot \nu d\sigma(x)  }{\int_{ \partial \Omega \cap B_\delta (x_0)}  |\Delta^\frac{m}{2} u_n|^2  \nu \cdot \nu(x_0)d\sigma(x) } \quad \text{ and } \quad y_n:= x_0 +\rho_n \nu(x_0).
\end{equation*}
Observe that $|y_n -x_0 | \le 2\delta$. Applying the Pohozaev identity of Lemma \ref{Pohozaev} on $\Omega_\delta = \Omega \cap B_\delta(x_0)$,  we obtain
\begin{equation}\label{our Pohozaev}
\begin{split}
\frac{1}{2}\int_{\partial \Omega_\delta}  |\Delta^\frac{m}{2} u_n|^2 &(x-y_n)\cdot \nu   \,d\sigma(x)  +\int_{\partial \Omega_\delta} f_n(x)d\sigma(x) \\ 
& =  \int_{\partial  \Omega_\delta} H_n(u_n(x)) (x-y_n) \cdot \nu d\sigma(x) - 2m \int_{ \Omega_\delta} H_n(u_n(x))dx, 
\end{split}
\end{equation}
where $H_n(t)= \frac{\lambda_n}{2\beta_n}e^{\beta_n t^2} + \frac{\alpha}{2}t^2$, and 
$$
f_n:= \sum_{j=0}^{m-1} (-1)^{m+j} \nu\cdot \bra{\Delta^{\frac{j}{2}} ((x-y_n)\cdot \nabla u_n) \Delta^{\frac{2m-j-1}{2}}u_n}.
$$  
Observe that the definition of $y_n$ implies 
\begin{equation}\label{int 0}
\int_{\partial \Omega \cap B_\delta (x_0)}  |\Delta^\frac{m}{2} u_n|^2 (x-y_n)\cdot \nu   \,d\sigma(x) =0, 
\end{equation}
and thus, by Lemma \ref{conv 0 boun}, we have
\begin{equation}\label{int 1}
 \int_{\partial \Omega_\delta} |\Delta^\frac{m}{2} u_n|^2  (x-y_n)\cdot \nu  \,d\sigma(x) = \int_{\Omega \cap  \partial B_\delta(x_0)} |\Delta^\frac{m}{2} u_n|^2  (x-y_n)\cdot \nu  \,d\sigma(x) = o(\mu_n^{-2}).
\end{equation}
Similarly, since $f_n= -  |\Delta^\frac{m}{2} u_n|^2 (x-y_n)\cdot \nu $ on $\partial \Omega \cap B_\delta(x_0)$, applying \eqref{int 0} and Lemma \ref{conv 0 boun}, we get
\begin{equation}\label{int 2}
 \int_{\partial \Omega_\delta} f_n(x)d\sigma(x) =   \int_{\Omega \cap \partial B_\delta(x_0)} f_n(x)d\sigma(x) = o(\mu_n^{-2}).
\end{equation}
Furthermore, we have 
\[\begin{split}
\int_{\partial \Omega_\delta}  e^{\beta_n u_n^2} (x-y_n)\cdot \nu  d\sigma(x)  &=  \int_{  \Omega \cap \partial B_\delta(x_0) }  e^{\beta_n u_n^2} (x-y_n)\cdot \nu  d\sigma (x) + \int_{\partial \Omega \cap B_\delta(x_0)} (x-y_n)\cdot \nu  d\sigma(x) \\
& = I_{\delta,n} + o(\mu_n^{-2}), 
 \end{split}
\]
where $I_{\delta,n} =  \int_{\partial \Omega_\delta} (x-y_n)\cdot \nu  d\sigma(x)=O(\delta)$ uniformly with respect to $n$. In particular,
\begin{equation}\label{int 3}
\begin{split}
 \int_{\partial \Omega_\delta} H_n(u_n(x)) (x-y_n)\cdot \nu  d\sigma(x)  &= \frac{\lambda_n }{2\beta_n} \int_{\partial \Omega_\delta}  e^{\beta_n u_n^2} (x-y_n)\cdot \nu  d\sigma(x) +  \frac{\alpha}{2} \int_{\Omega \cap \partial B_\delta(x_0) } u_n^2  (x-y_n)\cdot \nu d\sigma (x)\\ 
& =  \frac{\lambda_n}{2\beta_n}I_{\delta,n} +o(\mu_n^{-2}).\\
\end{split}
\end{equation}
Finally, we have 
\begin{equation} \label{int 4}
\begin{split}
 \int_{ \Omega_\delta} H_n(u_n(x))dx  &= \frac{\lambda_n }{2\beta_n} \int_{\Omega_\delta }  e^{\beta_n u_n^2} dx +  \frac{ \alpha}{2} \int_{ \Omega_\delta} u_n^2 dx \\ 
& = \frac{\lambda_n}{2\beta_n}  \int_{\Omega_\delta}e^{\beta_n u_n^2} dx+  o(\mu_{n}^{-2}). \\
\end{split}
\end{equation}
Therefore, \eqref{int 1}, \eqref{int 2}, \eqref{int 3}, \eqref{int 4} allow to rewrite  the identity in  \eqref{our Pohozaev} as 
\begin{equation}\label{Poho1}
\lambda_n \mu_n^2 \bra{ 2m \int_{\Omega_\delta}e^{\beta_n u_n^2} dx  -I_{\delta,n}} =o(1).
\end{equation}
Lemma \ref{conv 0 boun},  \eqref{extremal} and Lemma \ref{limsub}, assure
$$
 \int_{\Omega_\delta}e^{\beta_n u_n^2} dx = F_{\beta_n}(u_n) - \int_{\Omega\setminus B_\delta(x_0)} e^{\beta_n u_n^2}dx    \to S_{\alpha,\beta^*} - |\Omega \setminus B_\delta(x_0)|\ge S_{\alpha,\beta^*}- |\Omega| >0,
 $$
 as $n\to +\infty$. Then, for $\delta$ sufficiently small, the quantity $\int_{\Omega_\delta}e^{\beta_n u_n^2} dx - I_{n,\delta}$ is bounded away from $0$. Hence, the identity \eqref{Poho1} implies $\lambda_n \mu_n^2 \to 0$ and, since $I_{n,\delta}=O(\delta)$, 
\begin{equation}\label{Poho2}
\lambda_n \mu_n^2  \int_{\Omega_\delta}e^{\beta_n u_n^2} dx =o(1).
\end{equation}
But \eqref{Poho2} contradicts Remark \ref{rem integrals}, since for any large $R>0$ one has 
$$
\lambda_n \mu_n^2  \int_{\Omega_\delta}e^{\beta_n u_n^2} dx \ge  \lambda_n \mu_n^2  \int_{B_{R r_n }(x_n)}e^{\beta_n u_n^2} dx = 1 +O(R^{-2m}). 
$$
\end{proof}

\subsection{Neck analysis}
In this subsection, we complete the proof of Proposition \ref{big} by giving a sharp upper bound on $\dis{\frac{1}{\lambda_n \mu_n^2}}$. Let us fix a large $R>0$ and a small $\delta>0$ and  let us consider the annular region 
\begin{equation*}
A_n(R,\delta):=\{x \in \Omega : r_n R \le |x-x_n| \le \delta \},
\end{equation*}
where $r_n$ is given by \eqref{rn}. Note that, by Lemma \ref{no boundary}, we have  $A_n(R,\delta)\subseteq \Omega$, for any  $0<\delta<d(x_0,\partial \Omega)$ and any sufficiently large $n\in \N$. Our main idea is to compare the Dirichlet energy of $u_n$ on $A_n(R,\delta)$ with the energy of  the $m-$harmonic function 
\begin{equation*}
 \W_n(x):=-\frac{2m}{\beta^*\mu_n}\log |x-x_n|.  
\end{equation*}
As a consequence of Proposition \ref{conv eta} and \eqref{etan}, on $\partial B_{Rr_n}(x_n)$, we have 
$$
u_n(x) = \mu_n + \frac{\eta_0(\frac{x-x_n}{r_n})}{\mu_n} +o(\mu_n^{-1}) = \mu_n  -\frac{2m}{\beta^*\mu_n} \log \frac{R}{2} + \frac{O(R^{-2})}{\mu_n}+o(\mu_n^{-1}),
$$
as $n\to +\infty$. Similarly, using also \eqref{asym}, we find 
$$
\Delta^\frac{j}{2} u_n (x)= \frac{\Delta^\frac{j}{2} \eta_0(\frac{x-x_n}{r_n})}{r_n^j \mu_n}  + o(r_n^{-j}\mu_n^{-1}) =-\frac{2m K_{m,\frac{j}{2}}}{\beta^*r_n^j\mu_n R^{j}} e_{n,j}
+ \frac{O(R^{-j-2})}{r_n^j\mu_n}+o(r_n^{-j}\mu_n^{-1}),
$$
for any $1\le j \le 2m-1$, where $e_{n,j}:= e_j(x-x_n)$ with $e_j$ is as in  \eqref{ej}. 
The function $\W_n$ has an analog behaviour. Indeed, remembering the definition of $r_n$ in \eqref{rn}, we get
\begin{equation}\label{Wn1}
\W_n(x) = \frac{\beta_n}{\beta^*}  \mu_n - \frac{2m}{\beta^*\mu_n} \log R  + \frac{1}{\beta^* \mu_n} \log \bra{\omega_{2m} \lambda_n \mu_n^2}, 
\end{equation}
and, by \eqref{ej}, 
\begin{equation}\label{Wn2}
\Delta^\frac{j}{2} \W_n = -\frac{2m K_{m,\frac{j}{2}}}{\beta^*\mu_n r_n^j R^{j}}e_{n,j}, \qquad \mbox{ for any }\, 1\le j\le 2m-1,
\end{equation}
on $\partial B_{Rr_n}(x_n)$. We can so conclude that, as $n\to+\infty$, on $\partial  B_{R r_n}(x_n)$, we have the expansions
\begin{equation}\label{ex1}
u_n - \W_n =  \bra{1-\frac{\beta_n}{\beta^*} }\mu_n +\frac{1}{\beta^*\mu_n} \log \bra{ \frac{2^{2m}}{\omega_{2m} \lambda_n \mu_n^2}} + \frac{O(R^{-2})}{\mu_n}+o(\mu_n^{-1}),
\end{equation}
and
\begin{equation}\label{ex2}
\Delta^\frac{j}{2} (u_n -\W_n) =  \frac{O(R^{-j-2})}{r_n^j\mu_n}+o(r_n^{-j}\mu_n^{-1}),\qquad \text{ for any } 1\le j\le 2m-1.
\end{equation}
Similarly, on $\partial B_\delta (x_n)$, we can use Lemma \ref{conv Green} and Propositon \ref{prop green} to get
\begin{equation}\label{ex3}
u_n - \W_n =\frac{C_{\alpha, x_0} }{\mu_n} + \frac{O(\delta)}{\mu_n}+o(\mu_n^{-1}),
\end{equation}
and
\begin{equation}\label{ex4}
\Delta^\frac{j}{2} (u_n -\W_n)= \frac{O(1)}{\mu_n}+o(\mu_n^{-1}), \qquad \text{ for any } 1\le j\le 2m-1.
\end{equation} 
Here we have also used that $\frac{|x-x_n|}{|x-x_0|}\to 1$, uniformly on $\partial B_\delta(x_n)$. The asymptotic formulas  in \eqref{Wn1}-\eqref{ex4} allow to compare $\| \Delta^{\frac{m}{2}}u_n \|_{L^2(A_n(R,\delta))}$ and $\| \Delta^{\frac{m}{2}} \W_n \|_{L^2(A_n(R,\delta))}$. Since the quantity $\lambda_n \mu_n^2$ appears in \eqref{ex1}, this will result in the desired upper bound.  

\begin{lemma}\label{final} Under the assumptions of Proposition \ref{big}, we have 
$$
\lim_{n\to +\infty} \frac{1}{\lambda_n \mu_n^2} \le \frac{\omega_{2m}}{2^{2m}} e^{\beta^*\bra{C_{\alpha,x_0} -I_m}}. 
$$
\end{lemma}
\begin{proof}

First,  Young's inequality yields
\begin{equation}\label{first}
\begin{split}
\| \Delta^{\frac{m}{2}}u_n \|_{L^2(A_n(R,\delta))}^2 - \| \Delta^{\frac{m}{2}}\W_n \|_{L^2(A_n(R,\delta))}^2\ge 2  \int_{A_n(R,\delta)}  \Delta^\frac{m}{2} (u_n-\W_n)  \cdot \Delta^\frac{m}{2} \W_n dx. 
\end{split}
\end{equation}
Integrating by parts, the integral in the  RHS equals to
\begin{equation}\label{b0}
\int_{A_n(R,\delta)}  \Delta^\frac{m}{2} (u_n-\W_n)  \cdot \Delta^\frac{m}{2} \W_n dx = - \int_{\partial A_n(R,\delta)} \sum_{j=0}^{m-1} (-1)^{m+j} \nu\cdot \bra{\Delta^{\frac{j}{2}} (u_n-\W_n) \Delta^{\frac{2m-j-1}{2}}\W_n} d\sigma.
\end{equation}
Let us denote, $\Lambda_n:= \frac{2^{2m}}{\omega_{2m} \lambda_n \mu_n^2}$. On $\partial B_{R r_n}(x_n)$, by \eqref{Wn2}, \eqref{ex1}, \eqref{ex2}, and the explicit expression of $K_{\frac{2m-1}{2}}$ (see \eqref{Kml2}), we find
\begin{equation}\label{b1}
\begin{split}
(u_n -\W_n ) \Delta^{\frac{2m-1}{2}} \W_n \cdot \nu &=  -\frac{2m}{\beta^*}  \bra{1 - \frac{\beta_n}{\beta^*} +\frac{1}{\beta^* \mu_n^2} \log\bra{ \Lambda_n} +  \frac{O(R^{-2})}{\mu_n^2}+o(\mu_n^{-2})} \frac{K_{m,\frac{2m-1}{2}}}{(r_nR)^{2m-1}} \\
& =  \frac{(-1)^m}{\omega_{2m-1}(r_n R)^{2m-1}}  \bra{1 - \frac{\beta_n}{\beta^*} +\frac{1}{\beta^* \mu_n^2} \log\bra{ \Lambda_n} +  \frac{O(R^{-2})}{\mu_n^2}+o(\mu_n^{-2})},
\end{split}
\end{equation}
and, for $1\le j\le m-1$, 
\begin{equation}\label{b2}
\begin{split}
\Delta^\frac{j}{2}(u_n -\W_n ) \Delta^{\frac{2m-j-1}{2}} \W_n \cdot \nu &=  \bra{\frac{O(R^{-2})}{\mu_n^2} + o(\mu_n^{-2})}O(r_n R)^{1-2m}.
\end{split}
\end{equation}
Similarly, on $\partial B_{\delta}(x_0)$, \eqref{ej}, \eqref{ex3} and \eqref{ex4} yield 
\begin{equation}\label{b3}
\begin{split}
(u_n -\W_n ) \Delta^{\frac{2m-1}{2}} \W_n \cdot \nu  & =  \frac{(-1)^m}{\omega_{2m-1}\delta^{2m-1}}  \bra{\frac{C_{\alpha,x_0}}{\mu_n^2}+  \frac{O(\delta)}{\mu_n^2}+o(\mu_n^{-2})},
\end{split}
\end{equation}
and 
\begin{equation}\label{b4}
\begin{split}
\Delta^\frac{j}{2}(u_n -\W_n ) \Delta^{\frac{2m-j-1}{2}} \W_n \cdot \nu &= \bra{  \frac{O(1)}{\mu_n^2} + o(\mu_n^{-2})  }O(\delta^{1+j- 2m}),
\end{split}
\end{equation}
for any $1\le j\le m-1$.
Using \eqref{b1}, \eqref{b2}, \eqref{b3}, \eqref{b4}, we can rewrite \eqref{b0} as 
\[
\int_{A_n(R,\delta)}  \Delta^\frac{m}{2} (u_n-\W_n)  \cdot \Delta^\frac{m}{2} \W_n dx =  \Gamma_n+  \frac{O(R^{-2})}{\mu_n^2}+\frac{O(\delta)}{\mu_n^2}+o(\mu_n^{-2}),
\] 
with 
\begin{equation}\label{Gamman}
\Gamma_n := 1 - \frac{\beta_n}{\beta^*} +\frac{1}{\beta^* \mu_n^2} \log\bra{ \Lambda_n}  - \frac{C_{\alpha,x_0}}{\mu_n^2}.
\end{equation}
Therefore, \eqref{first} reads as 
\begin{equation}\label{lower}
\| \Delta^{\frac{m}{2}}u_n \|_{L^2(A_n(R,\delta))}^2 - \| \Delta^{\frac{m}{2}}\W_n \|_{L^2(A_n(R,\delta))}^2\ge 2\Gamma_n+  \frac{O(R^{-2})}{\mu_n^2}+\frac{O(\delta)}{\mu_n^2}+o(\mu_n^{-2}).
\end{equation}
We shall now compute the difference in the LHS of \eqref{lower} in a  precise way. Since $\|u_n\|_\alpha =1$, we have
\[
\| \Delta^{\frac{m}{2}}u_n \|_{L^2(A_n(R,\delta))}^2 =  1 + \alpha \|u_n\|^2_{L^2(\Omega)}  - \int_{\Omega\setminus B_\delta(x_0)} |\Delta^\frac{m}{2} u_n|^2 dx  - \int_{B_{r_n R}(x_n)} |\Delta^\frac{m}{2} u_n|^2 dx. \]
By Lemma \ref{conv Green} and Lemma \ref{int Green}, we infer 
$$
\|u_n\|_{L^2(\Omega)}^2 =  \frac{\|G_{\alpha,x_0}\|_{L^2(\Omega)}^2}{\mu_n^2} +o(\mu_n^{-2}), 
$$
and
$$ 
\int_{\Omega\setminus B_\delta(x_0)} |\Delta^\frac{m}{2} u_n |^2 dx  = \mu_n^{-2}\bra{ \alpha \|G_{\alpha,x_0} \|_{L^2(\Omega)}^2 - \frac{2m }{\beta^*} \log \delta + C_{\alpha,x_0} +H_m + O(\delta| \log \delta|)+o(1)}.
$$
Moreover, Proposition \ref{conv eta} and Lemma \ref{int eta0} imply 
\[
\int_{B_{r_nR}(x_n)}|\Delta^\frac{m}{2} u_n |^2 dx= \mu_n^{-2}\bra{\frac{2m}{\beta^*}\log \frac{R}{2}  +  I_m   -H_m + O(R^{-2}\log R )+o(1)}.
\]
Therefore, 
\[
\| \Delta^{\frac{m}{2}}u_n \|_{L^2(A_n(R,\delta))}^2 = 1+ \frac{2m}{\beta^*\mu_n^2} \log \frac{2\delta}{R} -\frac{C_{\alpha,x_0} +I_m}{\mu_n^2} +  \frac{O(R^{-2}\log R )}{\mu_n^2} +\frac{O(\delta |\log \delta| )}{\mu_n^2} +o(\mu_n^{-2}).
\]
The identity $\omega_{2m-1}\frac{2m}{\beta^*}K_{m,\frac{m}{2}}^2=1$ and a direct computation show that 
\[\begin{split}
 \| \Delta^{\frac{m}{2}}\W_n \|_{L^2(A_n(R,\delta))}^2 &=  \omega_{2m-1}\bra{\frac{2m K_{m,\frac{m}{2}}}{\beta^* \mu_n }}^2 \log \frac{\delta}{R r_n} \\
&= \frac{2m}{\beta^* \mu_n^2} \log \frac{\delta}{R} +\frac{\beta_n}{\beta^*}  +\frac{1}{\beta^*\mu_n^2}\log\bra{{\omega_{2m}\lambda_n \mu_n^2}}.
\end{split}
\]
Hence, 
\begin{equation}\label{upper}
\| \Delta^{\frac{m}{2}}u_n \|_{L^2(A_n(R,\delta))}^2 - \| \Delta^{\frac{m}{2}}\W_n \|_{L^2(A_n(R,\delta))}^2 = \Gamma_n -\frac{I_m}{\mu_n^2}  +\frac{O(R^{-2}\log R)}{\mu_n^2}+\frac{O(\delta |\log \delta|)}{\mu_n^2}+o(\mu_n^{-2}),
\end{equation}
with $\Gamma_n$ as in \eqref{Gamman}. Comparing \eqref{lower} and \eqref{upper}, we find the upper bound
\begin{equation}\label{quasi}
\Gamma_n \le -\frac{I_m}{\mu_n^2}  +\frac{O(R^{-2}\log R)}{\mu_n^2}+\frac{O(\delta| \log \delta|)}{\mu_n^2}+o(\mu_n^{-2}).
\end{equation}
Since $\beta_n <\beta^*$, the definition of $\Gamma_n$  in \eqref{Gamman} implies
\[
\Gamma_n \ge \frac{1}{\beta^* \mu_n^2} \log\bra{ \Lambda_n}  - \frac{C_{\alpha,x_0}}{\mu_n^2},
\]
Then, \eqref{quasi} yields
\[
 \log\bra{ \Lambda_n} \le \beta^* ( C_{\alpha,x_0} -I_m)  + O(R^{-2} \log R)+O(\delta|\log \delta|) + o(1). 
\]
Passing to the limit as $n \to +\infty$, $R\to +\infty$ and $\delta \to 0$, we can conclude
\[
\lim_{n\to +\infty}\Lambda_n \le e^{\beta^*\bra{ C_{\alpha,x_0} -  I_m}}.
\]
\end{proof}

We have so concluded the proof of Proposition \ref{big}, which follows directly from Lemma \ref{lemma crucial}, Lemma \ref{no boundary}, and Lemma \ref{final}.

\section{Test functions and the proof of Theorem \ref{main}}\label{sec test}
In this section, we complete the proof of Theorem \ref{main} by showing that the upper bound on $S_{\alpha,\beta^*}$, given in Proposition \ref{big}, cannot hold. Consequently, any sequence $u_n \in M_\alpha$ satisfying \eqref{extremal} must be uniformly bounded in $\Omega$. 

\begin{lemma}\label{polynomial}
For any $x_0\in \R^{2m}$, and  $\eps$,$R$,$\mu>0$, there exists a unique radially symmetric  polynomial $p_{\eps,R,\mu,x_0}$ such that 
\begin{equation}\label{derp}
\partial_\nu^i p_{\eps,R,\mu,x_0} (x)= -  \partial_\nu^i \bra{\mu^2+ \eta_0 \bra{ \frac{x -x_0}{\eps}} + \frac{2m}{\beta^*}\log |x-x_0| }   \quad \mbox{ on }\partial B_{\eps R} (x_0),
\end{equation}
for any $0\le i\le m-1$, where $\eta_0$ is as in \eqref{eta0}.  Moreover, 
$p_{\eps,R,\mu,x_0}$ has the form 
\begin{equation}\label{formp}
p_{\eps,R,\mu,x_0} (x) = -\mu^2 + \sum_{j=0}^{m-1} c_j(\eps,R) |x-x_0|^{2j},
\end{equation} 
with 
$$
c_0 (\eps,R) =  -\frac{2m}{\beta^*}\log (2\eps) +d_0(R)  \quad \mbox{and} \quad c_j(\eps,R) = \eps^{-{2j}} R^{-2j}  d_j (R), \quad 1\le j\le m-1, $$
where $d_j(R) =O(R^{-2}) $ as $R  \to +\infty, \text{ for } 0\le j \le m-1$.
\end{lemma}
\begin{proof}
We can construct $p_{\eps,R,\mu,x_0}$ in the following way. Let $ d_1(R)$,...,$ d_{m-1}(R)$ be the unique solution of the non-degenerate linear system
\begin{equation}\label{system}
\sum_{j=\sqbra{\frac{i+1}{2}}}^{m-1} \frac{(2j)!}{(2j-i)!}  d_j (R)=  \frac{2m}{\beta^*}(-1)^{i} (i-1)!- R^i \eta_0^{(i)} (R), \quad i=1,\ldots, m-1.
\end{equation}
Set also 
\begin{equation}\label{0}
\begin{split}
\tilde d_0(\eps, R,\mu):&= - \bra{\mu^2 +\eta_0(R) + \frac{2m}{\beta^*}\log (\eps R)}-\sum_{j=1}^{m-1} d_j(R),
\end{split}
\end{equation}
and 
$$
q(x):= \tilde{d}_0(\eps, R,\mu)+  \sum_{j=1}^{m-1}  d_j(R)|x|^{2j}.
$$
If we define $p_{\eps,R,\mu,x_0}(x):= q\bra{\frac{x-x_0}{\eps R}}$, then $p_{\eps,R,\mu,x_0}(x)$ satisfies \eqref{derp} for any $0\le i\le m-1$. 
Since, as $R\to +\infty$,
$$
\eta_0^{(i)} (R) = \frac{2m}{\beta^*}(-1)^{i} (i-1)! R^{-i} + O(R^{-i-2}), \quad \text{ for } 1\le i\le m-1,$$ 
and the system in \eqref{system} is nondegenerate, we find ${d}_j = O(R^{-2})$ as $R \to +\infty$ for $1\le j\le m-1$. Similarly, we have 
\[
\begin{split}
\tilde d_0 (\eps, R, \mu) &= -\mu^2 -\frac{2m}{\beta^*}\log (2\eps) + d_0(R),
\end{split}
\]
where 
\[
d_0(R):= -\eta_0(R)-\frac{2m}{\beta^*}\log \frac{ R}{2} - \sum_{j=1}^{m-1} d_j(R),
\]
and, by \eqref{0} and the asymptotic behavior at infinity of $\eta_0$,  $d_0(R)= O(R^{-2})$ as $R\to +\infty$. Then $p_{\eps,R,\mu,x_0}$ has the form \eqref{formp} with $c_0(\eps,R):= \tilde d_0(\eps, R,\mu) +\mu^2$ and
$c_j(\eps,R):= (\eps R)^{-2j}  d_j(R).$ 
\end{proof}

\begin{rem}\label{rem poly}
Observe that Lemma \ref{polynomial} gives 
$$
\left| p_{\eps,R,\mu,x_0} +\mu^2 +\frac{2m}{\beta^*}\log(2\eps) \right|\le  C R^{-2} \quad \mbox{ and } \quad  |\Delta^{\frac{m}{2}}p_{\eps,R,\mu,x_0}| \le C \eps^{-m} R^{-m-2} ,
$$
in $B_{\eps R}(x_0)$, where $C$ depends only on $m$.   
\end{rem}

\begin{prop}\label{proptest}
For any $x_0\in \Omega$, and $0\le \alpha<\lambda_1(\Omega)$, we have 
$$
S_{\alpha,\beta^*}> |\Omega|+ \frac{\omega_{2m}}{2^{2m}} e^{\beta^*\bra{C_{\alpha,x_0}-I_m}},
$$
where $C_{\alpha,x_0}$ and $I_m$ are respectively as in Proposition \ref{prop green} and \eqref{Im}. 
\end{prop}
\begin{proof}
We consider the function 
\[
u_{\eps,\alpha,x_0}(x):= \begin{Si}{cc}
\mu_\eps +\dfrac{\eta_0\bra{\frac{x-x_0}{\eps}}}{\mu_\eps} + \dfrac{C_{\alpha,x_0} +  \psi_{\alpha, x_0}(x) + p_{\eps} (x)}{\mu_\eps}  & \text{ for } |x-x_0| < \eps R_\eps,\\
\dfrac{G_{\alpha,x_0}(x)}{\mu_\eps}  & \text{ for } |x-x_0|\ge \eps R_\eps, 
\end{Si}
\]
where $\psi_{\alpha,x_0}$ is as in the expansion of $G_{\alpha,x_0}$ given in Proposition \ref{prop green}, $R_\eps= |\log \eps|$, $\mu_\eps$ is a constant that will be fixed later, and $p_\eps := p_{\eps,R_\eps,\mu_\eps,x_0}$ is the polynomial defined in  Lemma \ref{polynomial}. To simplify the notation, in this proof we will write $u_\eps$ in place of $u_{\eps,\alpha,x_0}$ without specifying the dependence on $\alpha$ and $x_0$.

Note that the choice of $p_\eps$ (specifically \eqref{derp}) implies that, for sufficiently  small $\eps$, $u_\eps \in H^m_0(\Omega)$.   Moreover, we can  write $u_\eps= \frac{\tilde u_\eps}{\mu_\eps}$, 
where 
\begin{equation}\label{testfunction}
 \tilde u_\eps (x) = \begin{Si}{cc} 
\eta_0\bra{\frac{x-x_0}{\eps}} + C_{\alpha,x_0} +\psi_{\alpha, x_0} (x)  +  p_\eps +\mu_\eps^2 & \text{ if } |x-x_0|< \eps R_\eps,\\
G_{\alpha,x_0} & \text{ if } |x-x_0|\ge \eps R_\eps, 
\end{Si}
\end{equation}
is a function that does not depend on the choice of $\mu_\eps$, because of Lemma \ref{polynomial}. In particular, if we fix  $\mu_\eps := \|\tilde u_\eps \|_\alpha$, we get  $\|u_\eps\|_\alpha=1$, and so  $u_\eps \in M_\alpha$. In order to  compute $F_{\beta^*}(u_\eps)$, we need  a precise expansion of $\mu_\eps$. Observe that, by Lemma \ref{int eta0}, the function $\eta_\eps(x):= \eta_0\bra{\frac{x-x_0}{\eps}}$, satisfies 
\begin{equation}\label{test1}
\begin{split}
\int_{B_{\eps R_\eps}(x_0)}|\Delta^\frac{m}{2}  \eta_\eps|^2 dx  &= \int_{B_{R_\eps}(0)} |\Delta^{\frac{m}{2}} \eta_0|^2 dx \\
&= \frac{2m}{\beta^*}\log \frac{R_\eps}{2} + I_m - H_m + O(R_\eps^{-2}\log R_\eps).
\end{split}
\end{equation}
Since $\psi_{\alpha, x_0}\in C^{2m-1}(\ov \Omega)$, we have
\begin{equation}\label{test2}
\int_{B_{\eps R_\eps}(x_0)} |\Delta^\frac{m}{2} \psi_{\alpha,x_0} |^2 dx= O(\eps^{2m} R_\eps^{2m}),
\end{equation}
Remark \ref{rem poly} gives  $|\Delta^\frac{m}{2}p_\eps| = O(\eps^{-m}R_\eps^{-m-2})$ in $B_{\eps R_\eps}(x_0)$.  Therefore, 
\begin{equation}\label{test3}
\int_{B_{\eps R_\eps}(x_0)} |\Delta^\frac{m}{2} p_\eps |^2 dx  = O( R_\eps^{-4}).
\end{equation}
Using H\"older's inequality, \eqref{test1} and \eqref{test2}, we find
\begin{equation}\label{test4}
\begin{split}
\int_{B_{\eps R}(x_0)}  \Delta^\frac{m}{2}\eta_\eps \cdot \Delta^\frac{m}{2}\psi_{\alpha,x_0} dx  &\le \| \Delta^\frac{m}{2}\eta_\eps \|_{L^2(B_{\eps R_\eps}(x_0))} \| \Delta^\frac{m}{2}\psi_{\alpha,x_0} \|_{L^2(B_{\eps R_\eps}(x_0))}\\
&=O(\eps^{m}R_\eps^{m} \log^\frac{1}{2} R_\eps).
\end{split}
\end{equation}
Similarly, by \eqref{test1}, \eqref{test2} and \eqref{test3}, we get
\begin{equation}\label{test5}
\int_{B_{\eps R_\eps}(x_0)} \Delta^{\frac{m}{2}} \eta_\eps \cdot  \Delta^\frac{m}{2} p_\eps  dx = O (R_\eps^{-2}\log^\frac{1}{2} R_\eps),
\end{equation}
and 
\begin{equation}\label{test6}
\int_{B_{\eps R_\eps}(x_0)}  \Delta^\frac{m}{2}p_\eps \cdot \Delta^\frac{m}{2}\psi_{\alpha,x_0} dx  = O(\eps^{m}R_\eps^{m-2}).
\end{equation}
By \eqref{test1}, \eqref{test2}, \eqref{test3}, \eqref{test4}, \eqref{test5} and \eqref{test6}, we infer
$$
\int_{B_{\eps R_\eps}(x_0)} |\Delta^\frac{m}{2} \tilde u_\eps|^2dx = \frac{2m}{\beta^*}\log \frac{R_\eps}{2} +I_m- H_m  + O(R_\eps^{-2}\log R_\eps). 
$$
Furthermore, applying Lemma \ref{int Green}, we have
\[
\begin{split}
\int_{\Omega \setminus B_{\eps R_\eps}(x_0)} |\Delta^\frac{m}{2} \tilde u_\eps|^2 dx &= \int_{\Omega \setminus B_{\eps R_\eps}(x_0)} |\Delta^\frac{m}{2} G_{\alpha,x_0}|^2 dx   \\ &= -\frac{2m}{\beta^*} \log (\eps R_\eps) + C_{\alpha,x_0} + H_m +\alpha \|G_{\alpha,x_0}\|_{L^2(\Omega)}^2 + O(\eps R_\eps |\log (\eps R_\eps)|).
\end{split}
\]
Hence,
\begin{equation}\label{test7}
\int_{\Omega} |\Delta^\frac{m}{2} \tilde u_\eps|^2dx = -\frac{2m}{\beta^*}\log (2\eps)+C_{\alpha,x_0} + I_m + \alpha \|G_{\alpha,x_0} \|_{L^2(\Omega)}^2 + O(R_\eps^{-2}\log R_\eps). 
\end{equation}
Finally, since  \eqref{testfunction} and Remark \ref{rem poly}, imply $\tilde u_\eps = O( |\log \eps|)$ on $B_{\eps R_\eps}(x_0)$, and since  $G_{\alpha,x_0}=O(|\log|x-x_0||)$ near $x_0$, we find 
\begin{equation}\label{test8}
\begin{split}
\| \tilde u_\eps\|_{L^2(\Omega)}^2 &= \|G_{\alpha,x_0}\|_{L^2(\Omega\setminus B_{\eps R_\eps})}^2 + O(\eps^{2m} R_\eps^{2m} \log^2\eps )\\
& = \|G_{\alpha,x_0}\|_{L^2(\Omega)}^2 + O(\eps^{2m} R_\eps^{2m} \log^2 \eps ).
\end{split}
\end{equation}
Therefore, using \eqref{test7} and \eqref{test8}, we obtain
\begin{equation}\label{mueps}
\mu_\eps^2=\| \tilde u_\eps\|^2_\alpha = -\frac{2m}{\beta^*}\log (2\eps)+C_{\alpha,x_0} + I_m  + O(R_\eps^{-2}\log R_\eps).
\end{equation} 
We can now estimate $F_{\beta^*}(u_\eps)$. On $B_{\eps R_\eps} (x_0)$, by definition of $u_\eps$, we get
\[
u_\eps ^2  \ge  \mu_\eps^2 + 2\bra{ \eta_0\bra{\frac{x-x_0}{\eps}} +C_{\alpha,x_0} +\psi_{\alpha,x_0}(x) + p_\eps(x)}.
\]
Then, Lemma \ref{polynomial},  Remark \ref{rem poly}, and \eqref{mueps}, give 
\[
\begin{split}
u_\eps ^2  \ge  -\frac{2m}{\beta^*} \log(2 \eps) +  2\eta_0\bra{\frac{x-x_0}{\eps}} + C_{\alpha,x_0} -I_m +O(R_\eps^{-2}\log R_\eps).
\end{split}
\]
Hence, using  a change of variables  and Lemma \ref{int eta0}, 
\begin{equation}\label{final1}
\begin{split}
\int_{B_{\eps R_\eps}(x_0)} e^{\beta^* u_\eps^2} dx & \ge  \frac{1}{2^{2m}} e^{\beta^* \bra{C_{\alpha,x_0}-  I_m }} (1 + O(R_\eps^{-2} \log R_\eps)) \int_{B_{R_\eps}(0)} e^{2\beta^* \eta_0} dy \\
& =\frac{\omega_{2m}}{2^{2m}} e^{\beta^* \bra{C_{\alpha,x_0}  -I_m} } +  O(R_\eps^{-2} \log R_\eps)). 
\end{split}
\end{equation}
Outside $B_{\eps R_\eps}(x_0)$, the basic inequality $e^{t^2}\ge 1+t^2$ gives
\begin{equation}\label{final2}
\begin{split}
\int_{\Omega \setminus B_{\eps R_\eps}(x_0)} e^{\beta^* u_\eps^2} dx  & = \int_{\Omega \setminus B_{\eps R_\eps}(x_0)} e^{\frac{\beta^*}{\mu_\eps^2} G_{\alpha,x_0}^2 } dx \\ &  \ge |\Omega| + \frac{\beta^*}{\mu_\eps^2} \|G_{\alpha,x_0}\| _{L^2(\Omega)}^2 + o(\mu_\eps^{-2}) + O(\eps^{2m} R_\eps^{2m}).
\end{split}
\end{equation}
Since $R_\eps = O(\mu_\eps^2 )$, by \eqref{final1} and \eqref{final2}, we  conclude that 
\[\begin{split}
F_{\beta^*} (u_\eps) &\ge |\Omega|  + \frac{\omega_{2m}}{2^{2m}} e^{\beta^* \bra{C_{\alpha,x_0} - I_m} } +  \frac{\beta^*}{\mu_\eps ^2} \|G_{\alpha,x_0}\| _{L^2(\Omega)}^2 + o(\mu_\eps^{-2}).
\end{split}\]
In particular, for sufficiently small $\eps$, we find
$$
S_{\alpha,\beta^*}\ge F_{\beta^*}(u_\eps)>|\Omega|  + \frac{\omega_{2m}}{2^{2m}} e^{\beta^* \bra{C_{\alpha,x_0} - I_m} }.
$$
\end{proof}

We can now prove Theorem \ref{main} using Proposition \ref{big} and Proposition \ref{proptest}. 

\begin{proof}[Proof of Theorem \ref{main}] \mbox{ }

\emph{1.} Let $\beta_n$, $u_n$ and $\mu_n$ be as in \eqref{betan}, \eqref{extremal}, \eqref{mun and xn} and \eqref{mun and xn2}. Since $\|u_n\|_{\alpha}=1$ and $0\le \alpha <\lambda_1(\Omega)$, $u_n$ is bounded in $H^m_0(\Omega)$. In particular, we can find a function $u_0\in H^m_0(\Omega)$ such that, up to subsequences,   $u_n\rw u_0$ in $H^m_0(\Omega)$ and $u_n \to u_0$ a.e. in $\Omega$. The  weak lower semicontinuity of $\|\cdot\|_\alpha$ implies that $u_0\in M_\alpha$. By Propositions \ref{big} and \ref{proptest}, we must have $\dis{\limsup_{n\to +\infty}\mu_n \le C}$. Then, Fatou's Lemma and the dominated convergence theorem imply respectively $F_{\beta^*}(u_0)<+\infty$ and $F_{\beta_n}(u_n)\to F_{\beta^*}(u_0)$.  Since, by Lemma \ref{limsub}, $u_n$ is  maximizing sequence for $S_{\alpha,\beta^*}$, we conclude that $S_{\alpha,\beta^*}=F_{\beta^*}(u_0)$. Then, $S_{\alpha,\beta^*}$ is finite and attained. 

\emph{2.}  
Clearly, if $\beta>\beta^*$, using \eqref{Adams}, we get
$$
S_{\alpha,\beta} \ge S_{0,\beta}  =  +\infty,  \qquad \text{ for any } \alpha\ge 0. 
$$
Assume now $\alpha\ge \lambda_1(\Omega)$ and $0\le \beta \le  \beta^*$. Let $\ph_1$ be  an eigenfuntion for $\lm $ on $\Omega$ corresponding to $\lambda_1(\Omega)$, i.e. a nontrivial solution of 
\[
\begin{Si}{cc}
\lm \ph_1 = \lambda_1(\Omega)\ph_1  & \text{ in } \Omega,\\
\ph_1 = \partial_{\nu}\ph_1 = \ldots= \partial_{\nu}^{m-1} \ph_1 =0 & \text{ on }\partial \Omega. 
\end{Si}
\]
Observe that, for any $t\in \R$,
$$
\|t\ph_1\|_{\alpha}^2 = t^2(\lambda_1(\Omega) -\alpha)\|\ph_1\|^2_{L^2} \le 0.
$$
In particular, $t\ph_1\in M_\alpha$. Then we have
$$
S_{\alpha,\beta} \ge F_{\alpha,\beta}(t\ph_1) \to +\infty, 
$$
as $t\to +\infty$. 
\end{proof}

\appendix
\renewcommand\thesection{}
\section{Appendix: Some elliptic estimates}
\renewcommand\thesection{\Alph{section}}

In this appendix, we recall some useful elliptic estimates which have been used several times throughout the paper. We start by recalling that $m-$harmonic functions are of class $C^\infty$ and that bounds on their $L^1$-norm give local uniform estimates  on all their derivatives. 

\begin{prop}\label{ell harm}
Let $\Omega \subseteq \R^N$ be a bounded open set. Then, for any $m\ge 1$, $l \in \N$, $\gamma\in (0,1)$, and any open set  $V\subset\subset\Omega$, there exists a constant $C=C(m,l,\gamma,V,\Omega)$ such that every m-harmonic function $u$ in $\Omega$ satisfies 
$$
\|u\|_{C^{l,\gamma}(V)} \le C \|u\|_{L^1(\Omega)}. 
$$
\end{prop}

Proposition \ref{ell harm} can be deduced e.g. from Proposition 12 in \cite{mar}, and its proof is based on Pizzetti's formula \cite{Piz}, which is a generalization of the standard mean value property for harmonic functions. 

\medskip

If $m\ge 2$, in general $m-$harmonic functions on a bounded open set $\Omega$ do not satisfy the maximum principle, unless $\Omega$ is one of the so called positivity preserving domains (balls are the simplest example). However, it is always true that the $C^{m-1}$ norm of a $m-$harmonic function can be controlled in terms of the $L^\infty$ norm of its derivatives on $\partial \Omega$. 

\begin{prop}\label{ell harm2}
Let $\Omega \subseteq \R^N$ be a smooth bounded open set. Then, there exists a constant $C=C(\Omega)>0$  such that
$$
\|u\|_{C^{m-1}(\Omega)} \le C \sum_{l=0}^{m-1} \|\nabla^l u\|_{L^\infty (\partial \Omega)},
$$ 
for any $m-$harmonic function $u\in C^{m-1}(\ov \Omega)$. 
\end{prop}

We recall now the main results concerning Schauder and $L^p$ elliptic estimates for $\lm$.

\begin{prop}[see Theorem 2.18 of \cite{GGS}]\label{ell shau}
Let $\Omega\subseteq \R^N$ be a bounded open set with smooth boundary, and take $k,m\in \N$, $k \ge 2m$, and $\gamma \in (0,1)$. If $u \in H^m(\Omega)$ is a weak solution of the problem
\begin{equation}\label{general prob}
\begin{Si}{cl}
\lm u = f & \mbox{ in } \Omega, \\
\partial_\nu^j u = h_j  & \mbox{ on }\partial \Omega, \; 0\le j\le m-1,
\end{Si}
\end{equation}
with $f\in C^{k-2m,\gamma}(\Omega)$ and $h_j \in C^{k-j,\gamma}(\partial \Omega)$, $0\le j\le m-1$, then $u\in C^{k,\gamma}(\Omega)$ and there exists a constant $C=C(\Omega,k,\gamma)$ such that 
$$
\|u\|_{C^{k,\gamma}(\Omega)} \le C \bra{ \| f\|_{C^{k-2m,\gamma}(\Omega)}  +  \sum_{j=0}^{m-1}  \|h_j\|_{C^{k-j,\gamma}(\partial \Omega)}}.
$$
\end{prop}

\begin{prop}[see Theorem 2.20 of \cite{GGS}]\label{ell zero}
Let $\Omega\subseteq \R^N$ be a bounded open set with smooth boundary, and take $m,k\in \N$, $k \ge 2m$, and $p >1$.  If $u \in H^m(\Omega)$ is a weak solution of  \eqref{general prob} with  $f\in W^{k-2m,p}(\Omega)$ and $h_j \in W^{k-j-\frac{1}{p},p}(\partial \Omega)$, $0\le j\le m-1$, then $u\in W^{k,p}(\Omega)$ and there exists a constant $C=C(\Omega,k,p)$ such that 
$$
\|u\|_{W^{k,p}(\Omega)} \le C \bra{ \| f\|_{W^{k-2m,p}(\Omega)}  +  \sum_{j=0}^{m-1}  \|h_j\|_{W^{k-j-\frac{1}{p},\gamma}(\partial \Omega)}}.
$$
\end{prop}

In the absence of boundary conditions one can obtain local estimates combining Propositions \ref{ell shau} and \ref{ell zero} with Proposition \ref{ell harm}. 
\begin{prop}\label{ell loc}
Let $\Omega\subseteq \R^N$ be a bounded open set with smooth boundary and take $m,k\in \N$, $k \ge 2m$, $p >1$. If  $f\in W^{k-2m,p}(\Omega)$ and $u$ is a weak solution of  $\lm u = f$ in $\Omega$, then $u\in W^{k,p}_{loc}(\Omega)$ and, for any open set $V\subset \subset \Omega$, there exists a constant $C=C(k,p,V,\Omega)$ such that  
$$
\|u\|_{W^{k,p}(V)} \le C \bra{ \| f\|_{W^{k-2m,p}(\Omega)}  +  \|u\|_{L^1(\Omega)}}. 
$$
Similarly, if $f\in C^{k-2m,\gamma}(\Omega)$ and $u$ is a weak solution of $\lm u =f$ in $\Omega$, then $u\in C^{k,\gamma}_{loc}(\Omega)$ and, for any open set $V\subset \subset \Omega$, there exists a constant $C=C(k,\gamma,V,\Omega)$ such that  
$$
\|u\|_{C^{k,\gamma}(V)} \le C \bra{ \| f\|_{C^{k-2m,\gamma}(\Omega)}  +  \|u\|_{L^1(\Omega)}}. 
$$
\end{prop}

In many cases, one has to deal with solutions of $\lm u =f$ in $\Omega$, with boundary conditions satisfied only on a subset of $\partial \Omega$. For instance, as a consequence of Proposition \ref{ell zero}, Green's representation formula, and the continuity of trace operators on $W^{m,1}(\Omega)$, one obtains the following Proposition. 

\begin{prop}\label{ell new}
Let $\Omega\subseteq \R^N$ be an open set with smooth boundary, and fix $x_0, x_1 \in \R^{2m}$ and $p>1$. For any $\delta,R >0$ such that $\Omega \cap B_R(x_1) \setminus B_{2\delta}(x_0)\neq \0$, there exists a constant $C=C(\Omega,x_0, x_1,\delta,R)$ such that every weak solution $u$ of problem \eqref{general prob}, with $f\in L^p(\Omega)$ and $h_j = 0$, $0\le j\le m-1$, satisfies 
$$
\|u\|_{W^{2m,p}(\Omega \cap B_R(x_1)\setminus B_{2\delta}(x_0))} \le C( \|f\|_{L^p(\Omega\cap B_{2R}(x_1)\setminus B_\delta(x_0))} + \|u\|_{W^{m,1}(\Omega\setminus B_{2R}(x_1)\cap B_\delta(x_0))}).
$$
\end{prop}

\begin{rem}\label{remdomains}
The constant $C$ appearing in Proposition \ref{ell new} depends on $\Omega$ only through the $C^{2m}$ norms of the local maps that define $B_{2R}(x_1)\cap \partial \Omega$. In particular, Proposition \ref{ell new} can be applied uniformly to sequences $\{\Omega_n\}_{n\in \N}$, which  converge in the $C^{2m}_{loc}$ sense to a limit domain $\Omega$. 
\end{rem}

The following Proposition holds only in the special case $m=1$. It gives a Harnack-type inequality  which is useful to control the local behavior of a sequence of solutions of $-\Delta u = f$, when the behavior at one point is known. 

\begin{prop}\label{ell use}
Let  $u_n \in H^1(B_R(0))$ be a sequence of weak solutions of $-\Delta u_n = f_n$ in $B_R(0)\subseteq \R^N$, $R>0$. Assume that  $f_n$ is bounded in $L^\infty(B_R(0))$, and  there exists  $C>0$ such that  $u_n\le C$ and $u_n(0)\ge-C$. Then, $u_n$ is bounded in $L^\infty(B_\frac{R}{2}(0))$.
\end{prop}
\begin{proof}
We write $u_n = v_n + h_n$, with $h_n$ harmonic in $B_R(0)$, and $v_n$ solving 
$$
\begin{Si}{cc}
\Delta v_n = f_n  &  \mbox{ in }B_R(0),\\
v_n =  0 & \mbox{ on }\partial B_R(0).  
\end{Si}
$$
By Proposition \ref{ell zero}, $v_n$ is bounded in $W^{2,p}(B_R(0))$,  for any $p>1$. In particular, it is bounded in $L^\infty(B_R(0))$. Then, we have $$h_n = u_n -v_n \le C +\|v_n\|_{L^\infty(B_R(0))} \le \tilde{C},$$  and 
$$h_n(0)= u_n(0) -v_n(0)\ge -C -\|v_n\|_{L^\infty(B_R(0))}\ge - \tilde{C}.$$ By the mean value property, for any  $x\in B_{\frac{R}{2}}(0)$, we get 
\[
\begin{split}
h_n(x) - \tilde{C} &  = \frac{2^N}{\omega_N R^N} \int_{B_{\frac{R}{2}}(x)} (h_n - \tilde{C}) dy \\ 
& \ge \frac{2^N}{\omega_N R^N} \int_{B_{R}(0)} (h_n - \tilde{C}) dy \\ 
&
= 2^N (h_n(0)-\tilde{C}) \\
& \ge -2^{N+1} \tilde{C}.
\end{split}
\]
Hence, $h_n$ is bounded in $L^\infty(B_\frac{R}{2}(0))$. 
\end{proof}

Finally, we recall some Lorentz-Zygmund type elliptic estimates. For any $\alpha\ge 0$, let $L(\log L)^\alpha$ be defined as the space 
\begin{equation}\label{zygm}
L (\log L)^\alpha=\cur{ f: \Omega \ra \R \text{ s.t. } f \text{ is measurable  and } \int_\Omega |f|\log^\alpha (2+|f|)  dx <+\infty},
\end{equation}
and endowed with the norm 
\begin{equation}\label{zygm norm}
\|f\|_{L(Log L)^\alpha}:= \int_\Omega |f|\log^\alpha (2+|f|)  dx. 
\end{equation}
Given $1<p< +\infty$ , and $1\le q\le +\infty$, let $L^{(p,q)}(\Omega)$ be  the Lorentz space 
\begin{equation}\label{Lor}
L^{(p,q)}(\Omega):= \{ u : \Omega \ra \R\, :\, u \text{ is measurable and } \|u\|_{(p,q)}<+\infty \},
\end{equation}
where 
\begin{equation}\label{Lor norm}
\|u\|_{(p,q)}:= \bra{\int_{0}^{|\Omega|} t^{\frac{q}{p}-1} u^{**}(t)^q dt}^\frac{1}{q},  \quad \text{ for } 1 \le q<+\infty,
\end{equation}
and
\begin{equation}
\|u\|_{{(p,\infty)}} = \sup_{t\in (0,|\Omega|)} t^\frac{1}{p} u^{**}(t),
\end{equation}
with 
\begin{equation}
u^{**}(t) := t^{-1} \int^t_0 u^*(s) ds,
\end{equation}
and 
\begin{equation}\label{Lor fin}
u^*(t):=\inf\{\lambda >0 \; :\;  |\{|u|>\lambda\} | \le t \}.
\end{equation}

Among the many properties of Lorentz spaces we recall the following H\"older-type inequality (see \cite{ONe}). 

\begin{prop}\label{HolLor}
Let $1<p,p'< +\infty$, $1\le q,q'\le +\infty$, be such that $\frac{1}{p}+\frac{1}{p'}= \frac{1}{q}+\frac{1}{q'}=1$. Then, for any $u\in L^{(p,q)}(\Omega)$, $v\in L^{(p',q')}(\Omega)$, we have
$$
\|u v\|_{L^1(\Omega)} \le \|u\|_{(p,q)} \|v\|_{(p',q')}.
$$
\end{prop}

As proved in Corollary 6.16 of \cite{BS} (see also Theorem 10 in \cite{mar}) one has the following: 

\begin{prop}\label{ell Lor}
Let $\Omega\subseteq \R^N$, $N\ge 2m$, be a bounded smooth domain and take $0\le \alpha \le 1$. If $f\in L(\log L)^\alpha$, and $u$ is a weak solution of \eqref{general prob}, then $\nabla^{2m-l}u \in L^{(\frac{N}{N-l}, \frac{1}{\alpha})}(\Omega)$, for any $1\le l\le 2m-1$. Moreover, there exists a constant $C=C(\Omega,l)>0$ such that 
$$
\|\nabla^{2m-l} u \|_{(\frac{N}{N-l},\frac{1}{\alpha})} \le C  \|f\|_{L(Log L)^\alpha}.  
$$ 
\end{prop}

Note that, if $\alpha=0$, we have $L(\log^\alpha L) = L^1(\Omega)$. Moreover,  $L^{(\frac{N}{N-l}, \frac{1}{\alpha})}(\Omega)= L^{(\frac{N}{N-l},\infty)}(\Omega)$ coincides with the weak $L^\frac{N}{N-l}$ space on $\Omega$. In particular,  $L^{(\frac{N}{N-l},\infty)}(\Omega)\subseteq L^p(\Omega)$ for any $1\le p< \frac{N}{N-l}$. Therefore, as a consequence of Proposition \ref{ell Lor}, we recover the following well known result, whose classical proof relies on Green's representation formula.  

\begin{prop}\label{ell L1}
Let $\Omega\subseteq \R^{N}$, $N\ge 2m$,  be a bounded smooth domain. Then, for any $1\le l\le 2m-1$ and $1\le p<\frac{N}{N-l}$, there exists a constant $C=C(p,l,\Omega)$ such that every weak solution of \eqref{general prob} with  $f\in L^1(\Omega)$ satisfies  
$$
\|\nabla^{2m-l} u \|_{L^p(\Omega)} \le C  \|f\|_{L^1(\Omega)}.  
$$ 
\end{prop}

\end{document}